\documentclass[11pt]{amsart} 
\usepackage{amssymb,amsmath,amsthm,amsfonts,latexsym,enumerate,graphicx,color,multicol,bbm,mathptmx,microtype,tikz, caption, subcaption,}

\usetikzlibrary{arrows,positioning}

\tikzset{
dot/.style={
  circle, draw=black, very thin, minimum size=0.05cm, fill=black,
  inner sep=0, outer sep=0} }

\definecolor{myurlcolor}{rgb}{0.6,0,0}
\definecolor{mycitecolor}{rgb}{0,0,0.8}
\definecolor{myrefcolor}{rgb}{0,0,0.8}
\usepackage[backref=slide]{hyperref}
\hypersetup{colorlinks,
linkcolor=myrefcolor,
citecolor=mycitecolor,
urlcolor=myurlcolor}

\newtheorem{theorem}{Theorem} 
\newtheorem{corollary}[theorem]{Corollary}

\newtheorem{lemma}[theorem]{Lemma}

\newtheorem{exam}{Example}

\newtheorem{defi}{Definition}

\newtheorem*{rem}{Remarks}

\begin{document}

\title[Span$_2\{\mathcal{C}\}$]{Building the Bicategory Span$_2(\mathcal{C})$}
\author{Franciscus ~Rebro}
\address{Department of Mathematics\\
University of California\\
Riverside CA 92521\\
USA\\}
\email{rebro@math.ucr.edu}
\subjclass[]{}

\begin{abstract}
Given any category $\mathcal{C}$ with pullbacks and a terminal object, we show that the data consisting of the objects of $\mathcal{C}$, the spans of $\mathcal{C}$, and the isomorphism classes of spans of spans of $\mathcal{C}$, forms a bicategory. We denote this bicategory $\mbox{Span}_2\{\mathcal{C}\}$, and show that this construction can be applied to give a bicategory of $n$-manifolds, cobordisms of those manifolds, and cobordisms of cobordisms.

\end{abstract}

\thanks{With heartfelt thanks to John Baez}

\maketitle


\section{Historical context and motivation}
As early as Jean B\'enabou's 1967 work \cite{Ben} introducing the concept of bicategory or weak 2-category, bicategories of spans have been studied. Informally, a span in some category is a pair of arrows with a common source. This simple concept can be seen as a kind of more flexible arrow, which enjoys the benefit of always being able to be turned around; as such, they have arisen in diverse areas of study, including general relativity and quantum mechanics. For instance, in general relativity one often considers the situation of a pair of $n$-dimensional spaces bridged by an $n+1$-dimensional spacetime, called a cobordism. There is a category $n\mbox{Cob}$ in which objects are $n$-manifolds (representing some space in various states) and morphisms are cobordisms of those manifolds (representing possible spacetimes bridging a `past space' to a `future space'). The cobordisms come equipped with a pair of inclusion arrows from the boundary spaces to the connecting spacetime, which makes this data what is called a \emph{cospan}. A cospan is simply a span in the opposite category, meaning cobordisms are certain spans in $\mbox{Diff}^{\mbox{\footnotesize op}}$ (here $\mbox{Diff}$ denotes the category of smooth manifolds and smooth maps between them). \\ \\
Quantum mechanics deals with the category $\mbox{Hilb}$ of Hilbert spaces and bounded linear operators, where the fact that the arrows always have linear-algebraic adjoints results in there being a \emph{dagger compact} full subcategory of finite dimensional Hilbert spaces. Khovanov in 2010 gave a categorified version of the Heisenberg algebra \cite{Kho}, and in a 2012 paper, Jeffrey Morton and Jamie Vicary \cite{MorVic} examined this idea from a different point of view. In their paper, they worked in a bicategory where the morphisms are spans and the 2-morphisms are spans of spans, and showed that their 2-morphisms could be used to express certain equations derived by Khovanov. \\ \\
Morton and Vicary began with a bicategory of groupoids, and constructed from it a symmetric monoidal bicategory of spans and spans of spans. Inspired by their work, we wish to show that an analogous, but different, construction holds given any \emph{category} with pullbacks, rather than a bicategory. More precisely, if $\mathcal{C}$ is a category with pullbacks, we will give a construction for a structure we denote $\mbox{Span}_2\{\mathcal{C}\}$, and prove that it is a bicategory. In $\mbox{Span}_2\{\mathcal{C}\}$, the objects are those of $\mathcal{C}$, the morphisms are spans in $\mathcal{C}$, and the 2-morphisms are isomorphism classes of spans of spans. 

\section{Preliminaries} Fix a category $\mathcal{C}$ with pullbacks and a terminal object. We begin with a number of lemmas concerning spans in $\mathcal{C}$, which will allow us to define a bicategory called Span$_2(\mathcal{C})$. Denote by $A,B,C,\ldots$ the objects of $\mathcal{C}$; these will be the objects of Span$_2(\mathcal{C})$. A \emph{span} in $\mathcal{C}$ from $A$ to $B$ is a diagram

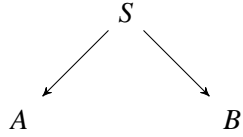
\begin{center}
\begin{tikzpicture}[->,>=stealth',shorten >=2pt,auto,node distance=2cm, main node/.style={color=black!70!black,font=\sffamily}]

  \node[main node] (1) {$S$};
  \node[color=black,font=\sffamily\bfseries] (2)[below left of=1]  {$A$};
   \node[color=black,font=\sffamily\bfseries] (3)[below right of=1]  {$B$};

  \path[every node/.style={font=\sffamily\small}]
 (1)edge  node[above](6){}  (2);
   \path[every node/.style={font=\sffamily\small}](1) 
   edge node [above] (7){}(3);

\end{tikzpicture}
\end{center}
\noindent
where $S$ is an object in $\mathcal{C}$. By an abuse of notation, we often let $S$ denote the entire diagram that is the span from $A$ to $B$; when more specificity is needed, we write $_{A}S_{B}$ for this span. We will usually not need to mention the particular arrows that constitute a span. The collection of all spans from $A$ to $B$ is denoted Span$(A,B)$. In the precise definition of Span$_2(\mathcal{C})$ to come, the collections Span$(A,B)$ will constitute the hom-categories, whose objects are referred to as the morphisms (or 1-cells) of Span$_2(\mathcal{C})$, and whose morphisms (up to isomorphism class) are referred to as the 2-morphisms (or 2-cells) of Span$_2(\mathcal{C})$. To make this precise, we now prove Span$(A,B)$ is a category under the appropriate notions of morphisms, composition, and identity morphisms.

\begin{lemma} Taking elements of Span$(A,B)$ as objects, isomorphism classes of spans between elements of Span$(A,B)$ as morphisms, the pullback of spans as composition of morphisms, and the identity span as the identity morphism, Span$(A,B)$ is a category.
\end{lemma}
\begin{proof} First we clarify our notion of morphism in Span$(A,B)$. Given two spans $S,T\in\,\mbox{Span}(A,B)$, a \emph{span of spans} $_{S}X_{T}$ is not only a span from $S$ to $T$, but one which makes the following diagram commute:

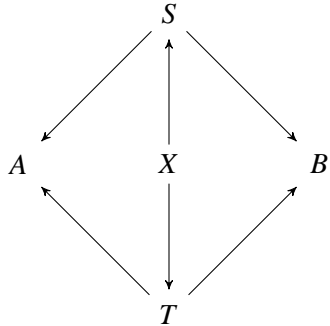
\begin{center}
\begin{tikzpicture}[->,>=stealth',shorten >=2pt,auto,node distance=2cm ,main node/.style={color=black!70!black,font=\sffamily\bfseries}]

  \node[main node] (1) {$S$};
  
              \node[color=black,font=\sffamily\bfseries] (2)[below of=1]  {$X$};  
              \node[main node] (3)[below of=2]  {$T$};

  \node[color=black,font=\sffamily\bfseries] (6)[left of=2]  {$A$};
   \node[color=black,font=\sffamily\bfseries] (7)[right of=2]  {$B$};

        \path[every node/.style={font=\sffamily\small}]
 (1)edge  node[above](25) {} (6);
   \path[every node/.style={font=\sffamily\small}](1) 
   edge node [above] (26){}(7);
  \path[every node/.style={font=\sffamily\small}]
 (3)edge  node[above](21){}  (6);
   \path[every node/.style={font=\sffamily\small}](3) 
   edge node [above] (22){}(7);

        \path[every node/.style={font=\sffamily\small}]
 (2)edge  node[above](25) {} (1);
   \path[every node/.style={font=\sffamily\small}](2) 
   edge node [above] (26){}(3);

\end{tikzpicture}
\end{center}

\noindent
Two spans of spans $_{S}X_{T}$ and $_{S}\tilde{X}_{T}$ are said to be \emph{isomorphic as spans of spans} if there is an isomorphism from $X$ to $\tilde{X}$ in $\mathcal{C}$ which makes the diagram below commute:

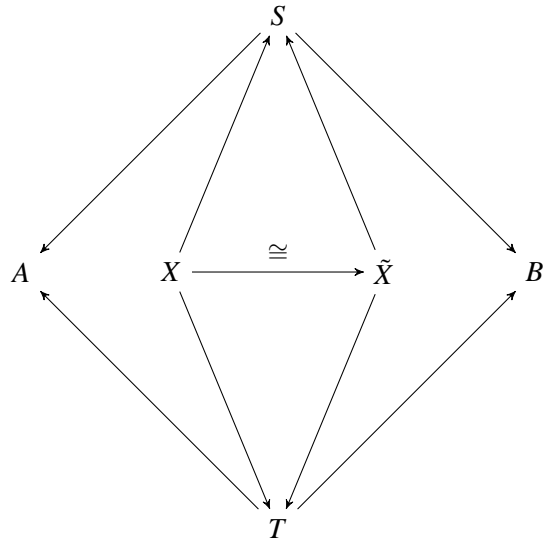
\begin{center}
\begin{tikzpicture}[->,>=stealth',auto,node distance=2cm,main node/.style={color=black!70!black,font=\sffamily\bfseries}]

  \node[main node] (1) {};
  
              \node[color=black,font=\sffamily\bfseries] (2)[below left of=1]  {$X$};  
              \node[main node] (3)[below right of=1]  {$\tilde{X}$};
              \node[main node] (4)[below right of=2]  {};

  \node[color=black,font=\sffamily\bfseries] (6)[left of=2]  {$A$};
   \node[color=black,font=\sffamily\bfseries] (7)[right of=3]  {$B$};    
 \node[color=black,font=\sffamily\bfseries] (8)[below of=4]  {$T$};    
 \node[color=black,font=\sffamily\bfseries] (9)[above of=1]  {$S$};

        \path[every node/.style={font=\sffamily\small}]
 (9)edge  node[above](25) {} (6);
   \path[every node/.style={font=\sffamily\small}](9) 
   edge node [above] (26){}(7);
  \path[every node/.style={font=\sffamily\small}]
 (8)edge  node[above](21){}  (6);
   \path[every node/.style={font=\sffamily\small}](8) 
   edge node [above] (22){}(7);

        \path[every node/.style={font=\sffamily\small}]
 (2)edge  node[above](25) {} (9);
   \path[every node/.style={font=\sffamily\small}](2) 
   edge node [above] (26){}(8);

        \path[every node/.style={font=\sffamily\small}]
 (3)edge  node[above](25) {} (9);
   \path[every node/.style={font=\sffamily\small}](3) 
   edge node [above] (26){}(8);

\path[every node/.style={font=\sffamily\small}]
 (2)edge  node[above](25) {$\cong$} (3);
   
\end{tikzpicture}
\end{center}

For our purpose of making Span$(A,B)$ into a category, if $X,\tilde{X}$ are isomorphic as spans of spans, then we declare $X$ and $\tilde{X}$ represent the same morphism from $S$ to $T$, and we write $X=\tilde{X}$. Having defined the objects and morphisms of Span$(A,B)$, we move on to composition of morphisms. Consider the diagram

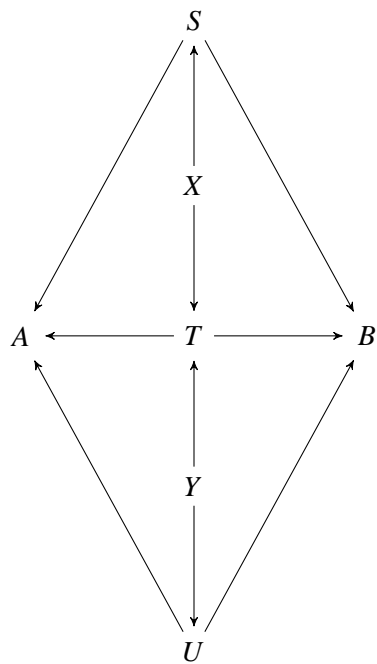
\begin{center}
\begin{tikzpicture}[->,>=stealth',shorten >=2pt,auto,node distance=2cm,main node/.style={color=black!70!black,font=\sffamily\bfseries}]

  \node[main node] (1) {$S$};
  
              \node[color=black,font=\sffamily\bfseries] (2)[node distance=2.2cm, below of=1]  {$X$};  
              \node[main node] (3)[node distance=2cm, below of=2]  {$T$};
\node[color=black,font=\sffamily\bfseries] (4)[node distance=2cm, below of=3]  {$Y$}; 
\node[main node] (5)[node distance=2.2cm, below of=4]  {$U$};

  \node[color=black,font=\sffamily\bfseries] (6)[node distance=2.3cm, left of=3]  {$A$};
   \node[color=black,font=\sffamily\bfseries] (7)[node distance=2.3cm, right of=3]  {$B$};

        \path[every node/.style={font=\sffamily\small}]
 (1)edge  node[above](25) {} (6);
   \path[every node/.style={font=\sffamily\small}](1) 
   edge node [above] (26){}(7);
  \path[every node/.style={font=\sffamily\small}]
 (3)edge  node[above](21){}  (6);
   \path[every node/.style={font=\sffamily\small}](3) 
   edge node [above] (22){}(7);
  \path[every node/.style={font=\sffamily\small}]
 (5)edge  node[above](21){}  (6);
   \path[every node/.style={font=\sffamily\small}](5) 
   edge node [above] (22){}(7);

        \path[every node/.style={font=\sffamily\small}]
 (2)edge  node[above](25) {} (1);
   \path[every node/.style={font=\sffamily\small}](2) 
   edge node [above] (26){}(3);
        \path[every node/.style={font=\sffamily\small}]
 (4)edge  node[above](25) {} (3);
   \path[every node/.style={font=\sffamily\small}](4) 
   edge node [above] (26){}(5);

\end{tikzpicture}
\end{center}
\noindent
Making use of the pullback of $X$ and $Y$ over $T$, we obtain a span of spans from $S$ to $U$:

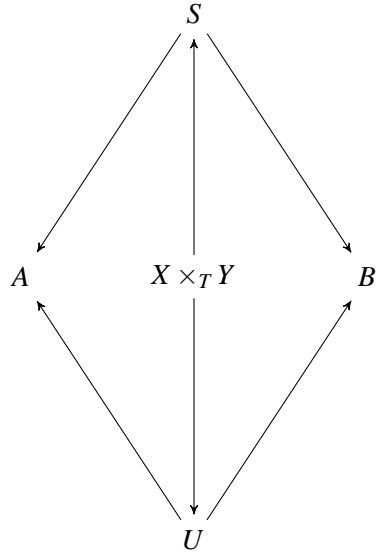
\begin{center}
\begin{tikzpicture}[->,>=stealth',shorten >=2pt,auto,node distance=1.75cm,main node/.style={color=black!70!black,font=\sffamily\bfseries}]

  \node[main node] (1) {$S$};
  
              \node[color=black,font=\sffamily\bfseries] (2)[below of=1]  {};  
              \node[main node] (3)[below of=2]  {$X \times_T Y$};
\node[color=black,font=\sffamily\bfseries] (4)[below of=3]  {}; 
\node[main node] (5)[below of=4]  {$U$};

  \node[color=black,font=\sffamily\bfseries] (6)[node distance=2.3cm, left of=3]  {$A$};
   \node[color=black,font=\sffamily\bfseries] (7)[node distance=2.3cm, right of=3]  {$B$};

        \path[every node/.style={font=\sffamily\small}]
 (1)edge  node[above](25) {} (6);
   \path[every node/.style={font=\sffamily\small}](1) 
   edge node [above] (26){}(7);
  \path[every node/.style={font=\sffamily\small}]
 (5)edge  node[above](21){}  (6);
   \path[every node/.style={font=\sffamily\small}](5) 
   edge node [above] (22){}(7);

        \path[every node/.style={font=\sffamily\small}]
 (3)edge  node[above](25) {} (1);
   \path[every node/.style={font=\sffamily\small}](3) 
   edge node [above] (26){}(5);
\end{tikzpicture}
\end{center}
\noindent
Here the arrow from $X \times_T Y$ to $S$ is the composite of the canonical arrow from $X \times_T Y$ to $X$ with the given arrow from $X$ to $S$ above, and the arrow from $X \times_T Y$ to $U$ is a similar composite. Observe that the diagram commutes as required: using that $X,Y$ are spans of spans, and the defining property of the pullback $X\times_T Y$, we have
$$X\times_T Y\to X\to S \to A= X\times_T Y \to X\to T \to A=X\times_T Y\to Y\to T \to A=X\times_T Y \to Y \to U \to A,$$
and similarly for the right hand side of the diagram. (The chains of arrows here stand for single composite arrows). Therefore we define $_{T}Y_{U} \circ _{S}X_{T}:=\,_{S}(X\times_T Y)_{U}$, and note that this is well-defined since any two pullbacks of a given cospan are canonically isomorphic, and such an isomorphism is also seen to be an isomorphism of spans of spans. We now show that this composition rule is associative. Consider the diagram underlying three composable spans of spans:

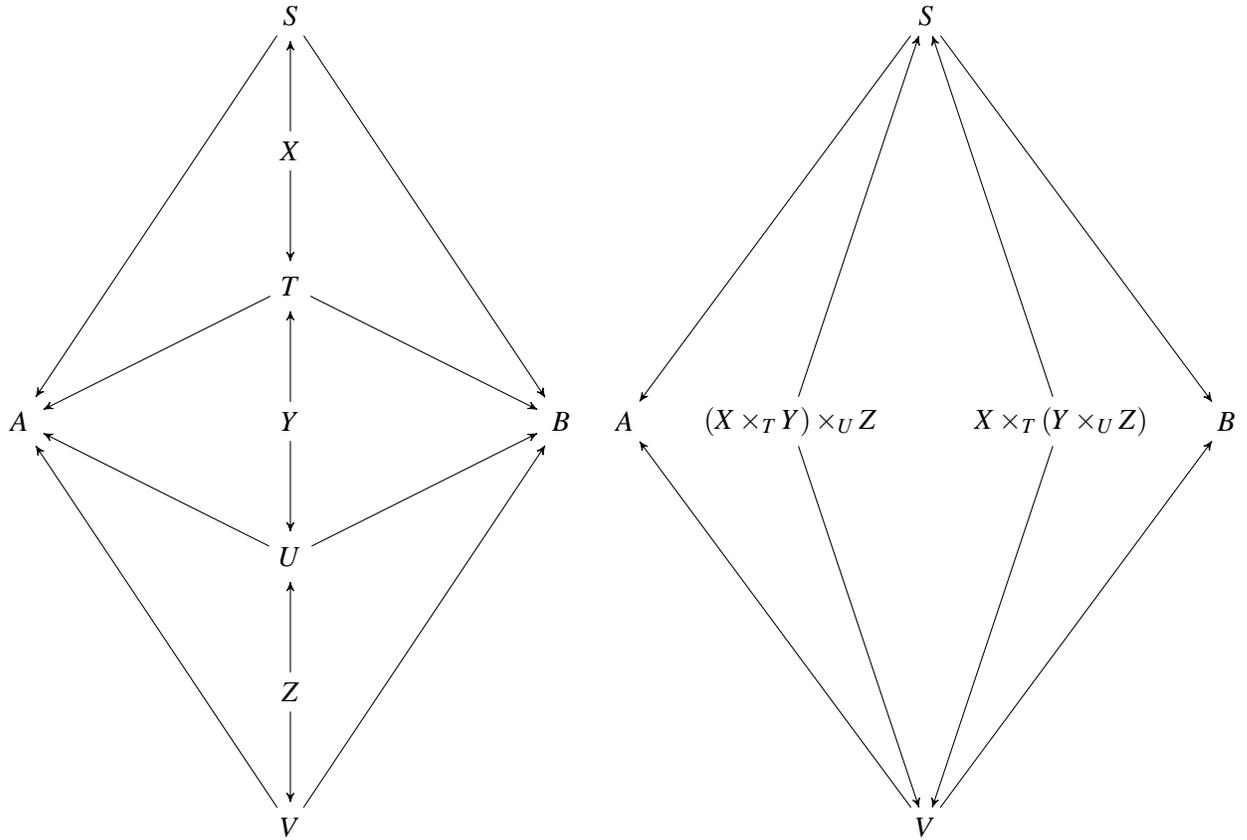
\begin{minipage}[b]{0.5\textwidth}
\centering
\begin{tikzpicture}[->,>=stealth',shorten >=2pt,auto,node distance=1.8cm,main node/.style={color=black!70!black,font=\sffamily\bfseries}]

  \node[main node] (1) {$S$};
  
              \node[color=black,font=\sffamily\bfseries] (2)[below of=1]  {$X$};  
              \node[main node] (3)[below of=2]  {$T$};
\node[color=black,font=\sffamily\bfseries] (4)[below of=3]  {$Y$}; 
\node[main node] (5)[below of=4]  {$U$};
\node[color=black,font=\sffamily\bfseries] (6)[below of=5]  {$Z$}; 
\node[main node] (7)[below of=6]  {$V$};

  \node[color=black,font=\sffamily\bfseries] (8)[left of=4]  {};
   \node[color=black,font=\sffamily\bfseries] (9)[right of=4]  {};       
  \node[color=black,font=\sffamily\bfseries] (10)[left of=8]  {$A$};
   \node[color=black,font=\sffamily\bfseries] (11)[right of=9]  {$B$};

        \path[every node/.style={font=\sffamily\small}]
 (1)edge  node[above](25) {} (10);
   \path[every node/.style={font=\sffamily\small}](1) 
   edge node [above] (26){}(11);
  \path[every node/.style={font=\sffamily\small}]
 (3)edge  node[above](21){}  (10);
   \path[every node/.style={font=\sffamily\small}](3) 
   edge node [above] (22){}(11);
  \path[every node/.style={font=\sffamily\small}]
 (5)edge  node[above](21){}  (10);
   \path[every node/.style={font=\sffamily\small}](5) 
   edge node [above] (22){}(11);
  \path[every node/.style={font=\sffamily\small}]
 (7)edge  node[above](21){}  (10);
   \path[every node/.style={font=\sffamily\small}](7) 
   edge node [above] (22){}(11);

        \path[every node/.style={font=\sffamily\small}]
 (2)edge  node[above](25) {} (1);
   \path[every node/.style={font=\sffamily\small}](2) 
   edge node [above] (26){}(3);
        \path[every node/.style={font=\sffamily\small}]
 (4)edge  node[above](25) {} (3);
   \path[every node/.style={font=\sffamily\small}](4) 
   edge node [above] (26){}(5);
        \path[every node/.style={font=\sffamily\small}]
 (6)edge  node[above](25) {} (5);
   \path[every node/.style={font=\sffamily\small}](6) 
   edge node [above] (26){}(7);
   
\end{tikzpicture}

\end{minipage}
\begin{minipage}[b]{0.5\textwidth}
\centering

\begin{tikzpicture}[->,>=stealth',auto,node distance=1.8cm,main node/.style={color=black!70!black,font=\sffamily\bfseries}]

  \node[main node] (1) {$S$};
  
              \node[color=black,font=\sffamily\bfseries] (2)[below of=1]  {};  
              \node[main node] (3)[below of=2]  {};
\node[color=black,font=\sffamily\bfseries] (4)[below of=3]  {}; 
\node[main node] (5)[below of=4]  {};
\node[color=black,font=\sffamily\bfseries] (6)[below of=5]  {}; 
\node[main node] (7)[below of=6]  {$V$};

  \node[color=black,font=\sffamily\bfseries] (8)[node distance=1.8cm, left of=4]  {$(X\times_T Y)\times_U Z$};
   \node[color=black,font=\sffamily\bfseries] (9)[node distance=1.8cm, right of=4]  {$X\times_T(Y\times_U Z)$};       
  \node[color=black,font=\sffamily\bfseries] (10)[node distance=2.2cm, left of=8]  {$A$};
   \node[color=black,font=\sffamily\bfseries] (11)[node distance=2.2cm, right of=9]  {$B$};

        \path[every node/.style={font=\sffamily\small}]
 (1)edge  node[above](25) {} (10);
   \path[every node/.style={font=\sffamily\small}](1) 
   edge node [above] (26){}(11);
  \path[every node/.style={font=\sffamily\small}]
 (7)edge  node[above](21){}  (10);
   \path[every node/.style={font=\sffamily\small}](7) 
   edge node [above] (22){}(11);
        \path[every node/.style={font=\sffamily\small}]
 (8)edge  node[above](25) {} (1);
   \path[every node/.style={font=\sffamily\small}](8) 
   edge node [above] (26){}(7);
        \path[every node/.style={font=\sffamily\small}]
 (9)edge  node[above](25) {} (1);
   \path[every node/.style={font=\sffamily\small}](9) 
   edge node [above] (26){}(7);

\end{tikzpicture}
\end{minipage}

 One can first take the pullback of $X$ and $Y$ over $T$, and then take the pullback of the result with $Z$ over $U$, to obtain $(X\times_T Y)\times_U Z$; similarly, one can form $X\times_T (Y\times_U Z)$. Both of these are morphisms from $S$ to $V$, and we will now show they are the same morphism, i.e. they are isomorphic as spans of spans. In the redrawn diagram below, $L$ denotes the limit of the base subdiagram determined by $_{S}X_{T}$, $_{T}Y_{U}$, and $_{U}Z_{V}$.

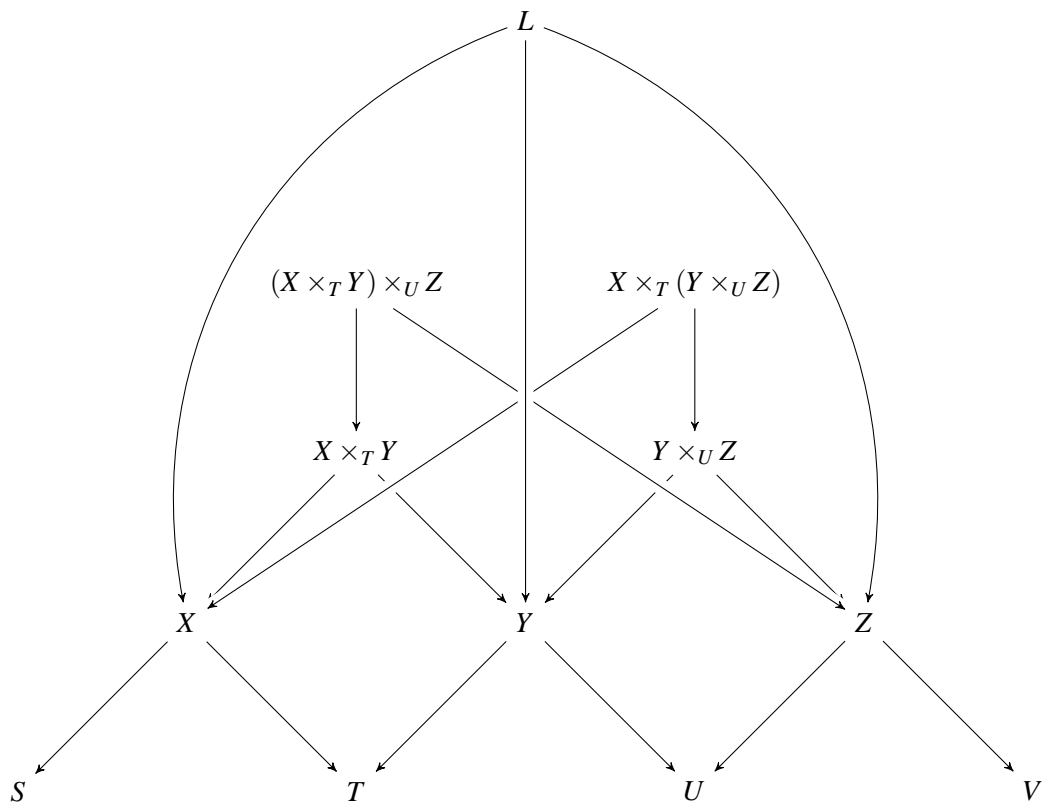
\begin{center}
\begin{tikzpicture}[->,>=stealth',node distance=2.25cm, auto]
  \node (S) {$S$};
  \node (X) [right of=S,above of=S] {$X$};
  \node (T) [below of=X,right of=X] {$T$};
  \node (Y) [above of=T,right of=T] {$Y$};
  \node (U) [below of=Y,right of=Y] {$U$};
  \node (Z) [above of=U,right of=U] {$Z$};
  \node (V) [below of=Z,right of=Z] {$V$};
  \node (XtimesY) [above of=X,right of=X] {$X\times_T Y$};
  \node (YtimesZ) [above of=Y,right of=Y] {$Y\times_U Z$};
  \node (XxYxz) [above of=XtimesY] {$(X\times_T Y)\times_U Z$};
  \node (xxYxZ) [above of=YtimesZ] {$X\times_T (Y\times_U Z)$};
  \node (L) [node distance=8cm, above of=Y] {$L$};
  \draw[->] (X) to node {} (S);
  \draw[->] (X) to node {} (T);
  \draw[->] (Y) to node {} (T);
  \draw[->] (Y) to node {} (U);
  \draw[->] (Z) to node {} (U);
  \draw[->] (Z) to node {} (V);
  \draw[->] (XtimesY) to node {} (X);
  \draw[->] (XtimesY) to node {} (Y);
  \draw[->] (YtimesZ) to node {} (Y);
  \draw[->] (YtimesZ) to node {} (Z);
  \draw[->] (XxYxz) to node {} (XtimesY);
  \draw[->] (XxYxz) to node {} (Z);
  \draw[->] (xxYxZ) to node {} (YtimesZ);
  \draw[->] (xxYxZ) to node {} (X);
  \draw[->,bend right=40] (L) to node {} (X);
  \draw[->] (L) to node {} (Y);
  \draw[->,bend left=40] (L) to node {} (Z);
  \draw[draw=white, -, line width=6pt] (L) -- (Y);
  \draw[draw=white, -, line width=6pt] (XxYxz) -- (Z);
  \draw[draw=white, -, line width=6pt] (xxYxZ) -- (X);
  \draw[->] (XxYxz) to node {} (Z);
  \draw[->] (xxYxZ) to node {} (X);
   \draw[draw=white, -, line width=6pt] (L) -- (Y);
  \draw[->] (L) to node {} (Y);
\end{tikzpicture}
\end{center}
Observe that $(X\times_T Y)\times_U Z$ is a cone on the base diagram of $X$, $Y$, and $Z$ that $L$ is built on - any two paths (chains of arrows) from $(X\times_T Y)\times_U Z$ to $T$ or to $U$ are equal, by definition of the pullback construction. The same is true for $X\times_T (Y\times_U Z)$ by the symmetry of the diagram. Therefore, by the definition of limit, there exist unique arrows $j$ and $k$ which make the whole diagram commute:

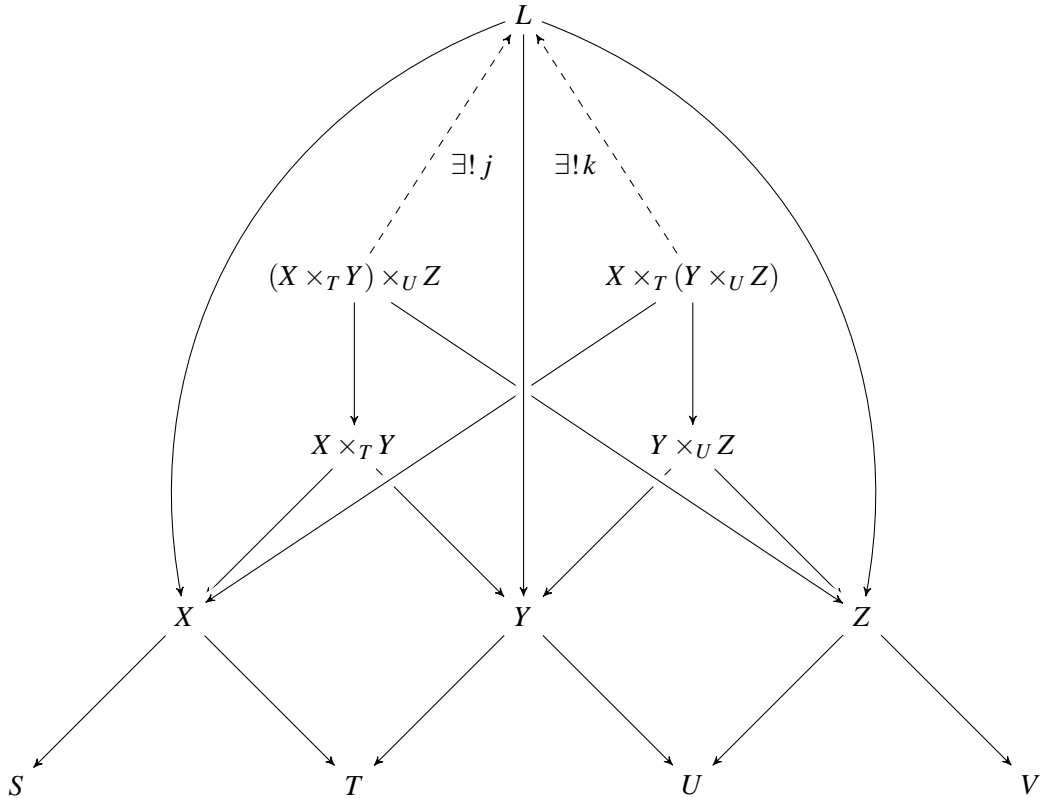
\begin{center}
\begin{tikzpicture}[->,>=stealth',node distance=2.25cm, auto]
  \node (S) {$S$};
  \node (X) [right of=S,above of=S] {$X$};
  \node (T) [below of=X,right of=X] {$T$};
  \node (Y) [above of=T,right of=T] {$Y$};
  \node (U) [below of=Y,right of=Y] {$U$};
  \node (Z) [above of=U,right of=U] {$Z$};
  \node (V) [below of=Z,right of=Z] {$V$};
  \node (XtimesY) [above of=X,right of=X] {$X\times_T Y$};
  \node (YtimesZ) [above of=Y,right of=Y] {$Y\times_U Z$};
  \node (XxYxz) [above of=XtimesY] {$(X\times_T Y)\times_U Z$};
  \node (xxYxZ) [above of=YtimesZ] {$X\times_T (Y\times_U Z)$};
  \node (L) [node distance=8cm, above of=Y] {$L$};
  \draw[->] (X) to node {} (S);
  \draw[->] (X) to node {} (T);
  \draw[->] (Y) to node {} (T);
  \draw[->] (Y) to node {} (U);
  \draw[->] (Z) to node {} (U);
  \draw[->] (Z) to node {} (V);
  \draw[->] (XtimesY) to node {} (X);
  \draw[->] (XtimesY) to node {} (Y);
  \draw[->] (YtimesZ) to node {} (Y);
  \draw[->] (YtimesZ) to node {} (Z);
  \draw[->] (XxYxz) to node {} (XtimesY);
  \draw[->] (XxYxz) to node {} (Z);
  \draw[->] (xxYxZ) to node {} (YtimesZ);
  \draw[->] (xxYxZ) to node {} (X);
  \draw[->,bend right=40] (L) to node {} (X);
  \draw[->] (L) to node {} (Y);
  \draw[->,bend left=40] (L) to node {} (Z);
  \draw[draw=white, -, line width=6pt] (L) -- (Y);
  \draw[draw=white, -, line width=6pt] (XxYxz) -- (Z);
  \draw[draw=white, -, line width=6pt] (xxYxZ) -- (X);
  \draw[->] (XxYxz) to node {} (Z);
  \draw[->] (xxYxZ) to node {} (X);
   \draw[draw=white, -, line width=6pt] (L) -- (Y);
  \draw[->] (L) to node {} (Y);
  \draw[->,dashed] (XxYxz) to node [swap] {$\exists !\,j$} (L);
  \draw[->,dashed] (xxYxZ) to node {$\exists !\,k$} (L);
\end{tikzpicture}
\end{center}
Next observe that $L$ is a cone on the (cospan) subdiagram determined by $X\times_T Y$, simply because it is the limit of a strictly larger diagram. Similarly, $L$ is a cone on the cospan determined by $Y\times_U Z$. Therefore, by the definition of the pullback construction, there are unique arrows $f$ and $g$ making these subdiagrams commute:

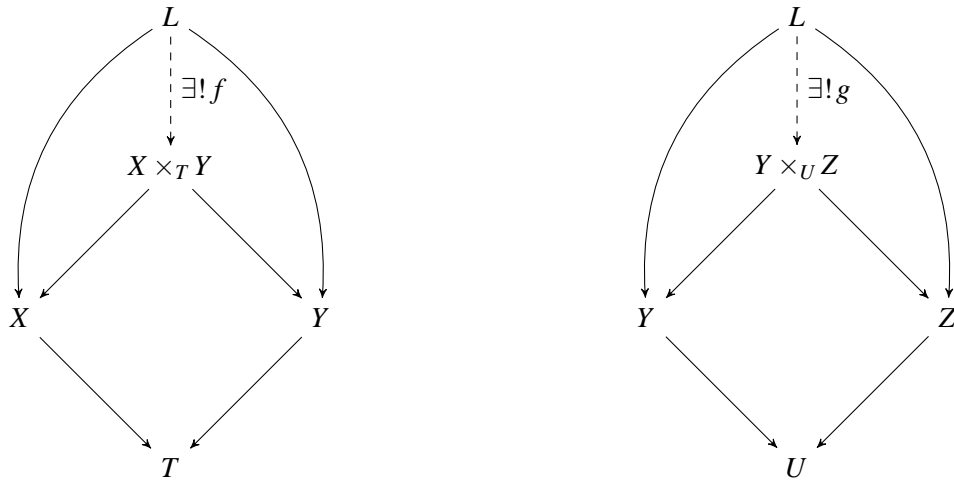
\begin{minipage}[b]{0.5\textwidth}
\centering
\begin{tikzpicture}[->,>=stealth',node distance=2cm, auto]
  \node (T) {$T$};
  \node (X) [left of=T,above of=T] {$X$};
  \node (Y) [right of=T,above of=T] {$Y$};
  \node (XxY) [right of=X,above of=X] {$X\times_T Y$};
  \node (L) [above of=XxY] {$L$};
  \draw[->] (X) to node {} (T);
  \draw[->] (Y) to node {} (T);
  \draw[->] (XxY) to node {} (X);
  \draw[->] (XxY) to node {} (Y);
 \draw[->,bend right] (L) to node {} (X);
  \draw[->,bend left] (L) to node {} (Y);
  \draw[->,dashed] (L) to node {$\exists !\,f$} (XxY);
\end{tikzpicture}

\end{minipage}
\begin{minipage}[b]{0.5\textwidth}
\centering
\begin{tikzpicture}[->,>=stealth',node distance=2cm, auto]
  \node (U) {$U$};
  \node (Y) [left of=U,above of=U] {$Y$};
  \node (Z) [right of=U,above of=U] {$Z$};
  \node (YxZ) [right of=Y,above of=Y] {$Y\times_U Z$};
  \node (L) [above of=YxZ] {$L$};
  \draw[->] (Y) to node {} (U);
  \draw[->] (Z) to node {} (U);
  \draw[->] (YxZ) to node {} (Y);
  \draw[->] (YxZ) to node {} (Z);
 \draw[->,bend right] (L) to node {} (Y);
  \draw[->,bend left] (L) to node {} (Z);
  \draw[->,dashed] (L) to node {$\exists !\,g$} (YxZ);
\end{tikzpicture}
\end{minipage}
Now using $f$ and $g$ we see that $L$ is also a cone on two other subdiagrams involving the iterated pullbacks $(X\times_T Y)\times_U Z$ and $X\times_T (Y\times_U Z)$, providing two more unique arrows that make these commute:

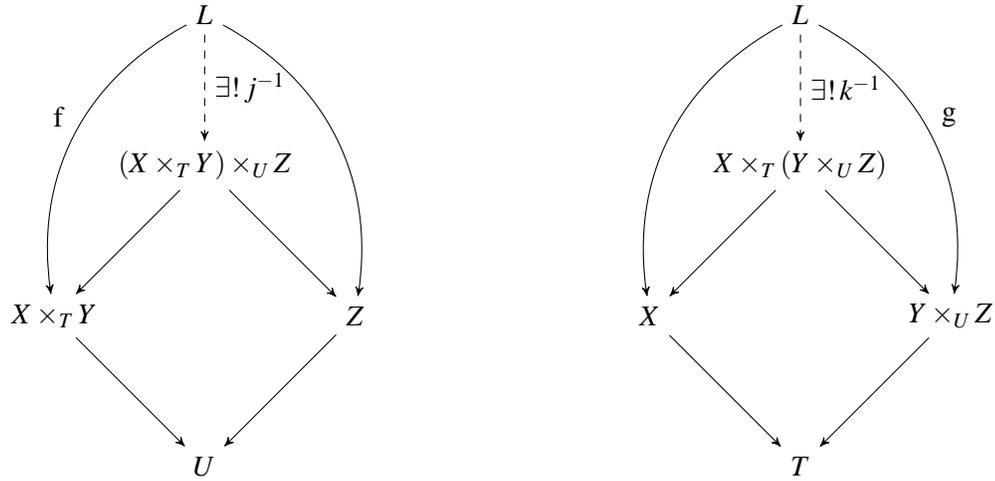
\begin{minipage}[b]{0.5\textwidth}
\centering
\begin{tikzpicture}[->,>=stealth',node distance=2cm, auto]
  \node (U) {$U$};
  \node (XxY) [left of=U,above of=U] {$X\times_T Y$};
  \node (Z) [right of=U,above of=U] {$Z$};
  \node (XxYxz) [right of=XxY,above of=XxY] {$(X\times_T Y)\times_U Z$};
  \node (L) [above of=XxYxz] {$L$};
  \draw[->] (XxY) to node {} (U);
  \draw[->] (Z) to node {} (U);
  \draw[->] (XxYxz) to node {} (XxY);
  \draw[->] (XxYxz) to node {} (Z);
 \draw[->,bend right=35] (L) to node [swap] {f} (XxY);
  \draw[->,bend left=35] (L) to node {} (Z);
  \draw[->,dashed] (L) to node {$\exists !\,j^{-1}$} (XxYxz);
\end{tikzpicture}

\end{minipage}
\begin{minipage}[b]{0.5\textwidth}
\centering
\begin{tikzpicture}[->,>=stealth',node distance=2cm, auto]
   \node (T) {$T$};
  \node (X) [left of=T,above of=T] {$X$};
  \node (YxZ) [right of=T,above of=T] {$Y\times_U Z$};
  \node (xxYxZ) [right of=X,above of=X] {$X\times_T (Y\times_U Z)$};
  \node (L) [above of=xxYxZ] {$L$};
  \draw[->] (X) to node {} (T);
  \draw[->] (YxZ) to node {} (T);
  \draw[->] (xxYxZ) to node {} (X);
  \draw[->] (xxYxZ) to node {} (YxZ);
 \draw[->,bend right=35] (L) to node {} (X);
  \draw[->,bend left=35] (L) to node {g} (YxZ);
  \draw[->,dashed] (L) to node {$\exists !\,k^{-1}$} (xxYxZ);
\end{tikzpicture}
\end{minipage}
We have called these two new arrows $j^{-1}$ and $k^{-1}$ because, as we will shortly prove, they are indeed inverses of $j$ and $k$ respectively. Consider the following diagram, in which the two arrows arriving at $X\times_T Y$ are the same, and so are the two arriving at $Z$:

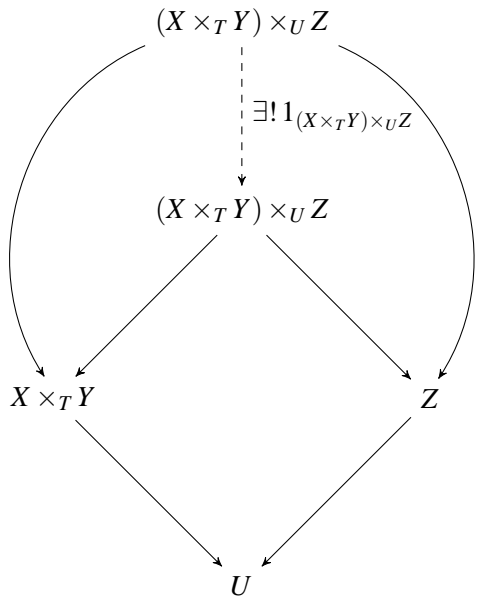
\begin{center}
\begin{tikzpicture}[->,>=stealth',node distance=2.5cm, auto]
\node (XxY) {$X\times_T Y$};
\node (XxYxz) [above of=XxY,right of=XxY] {$(X\times_T Y)\times_U Z$};
\node (Z) [below of=XxYxz,right of=XxYxz] {$Z$};
\node (XxYxz1) [above of=XxYxz] {$(X\times_T Y)\times_U Z$};
\node (U) [below of=Z,left of=Z] {$U$};
\draw[->] (XxYxz) to node {} (XxY);
\draw[->] (XxYxz) to node {} (Z);
\draw[->,dashed] (XxYxz1) to node {$\exists !\,1_{(X\times_T Y)\times_U Z}$} (XxYxz);
\draw[->,bend right=50] (XxYxz1) to node {} (XxY);
\draw[->,bend left=50] (XxYxz1) to node {} (Z);
\draw[->] (XxY) to node {} (U);
\draw[->] (Z) to node {} (U);
\end{tikzpicture}
\end{center}
By definition of the pullback, there is a unique arrow from $(X\times_T Y)\times_U Z$ to itself which makes the above diagram commute, and the identity on $(X\times_T Y)\times_U Z$ is an arrow which trivially satisfies that role. However, due to the commuting properties of the arrows $j$ and $j^{-1}$, one easily sees that the composite $j^{-1}\circ j$ also makes this diagram commute. Therefore, $j^{-1}\circ j=1_{(X\times_T Y)\times_U Z}$. We must also see that $j\circ j^{-1}=1_L$, which is accomplished in a similar fashion, by considering

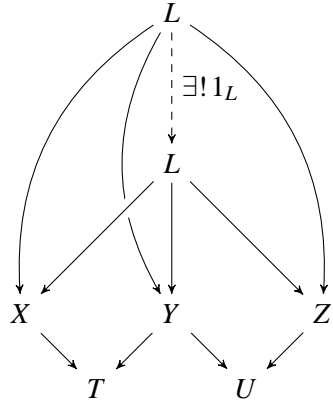
\begin{center}
\begin{tikzpicture}[->,>=stealth',node distance=2cm, auto]
\node (X) {$X$};
\node (L) [above of=X,right of=X] {$L$};
\node (Y) [below of=L] {$Y$};
\node (Z) [below of=L,right of=L] {$Z$};
\node (T) [node distance=1cm, below of=X, right of=X] {$T$};
\node (U) [node distance=1cm, below of=Y, right of=Y] {$U$};
\node (L1) [above of=L] {$L$};
\draw[->] (X) to node {} (T);
\draw[->] (Y) to node {} (T);
\draw[->] (Y) to node {} (U);
\draw[->] (Z) to node {} (U);
\draw[->] (L) to node {} (X);
\draw[->] (L) to node {} (Y);
\draw[->] (L) to node {} (Z);
\draw[->,dashed] (L1) to node {$\exists !\,1_L$} (L);
\draw[->,bend right] (L1) to node {} (X);
\draw[->,bend right] (L1) to node {} (Y);
\draw[draw=white, -, line width=6pt] (L) -- (X);
\draw[->] (L) to node {} (X);
\draw[->,bend left] (L1) to node {} (Z);
\end{tikzpicture}
\end{center}
Here the two arrows arriving at $X$ are the same, as are the two arrows arriving at $Y$ and at $Z$. By the definition of limit, this gives a unique arrow from $L$ to itself that makes the diagram commute, and the identity $1_L$ trivially satisfies this role. Observe that the commuting properties of $j^{-1}$ and $j$ imply that $j\circ j^{-1}$ also satisfies this role, and hence $j\circ j^{-1}=1_L$. We thus see that $j$ and $j^{-1}$ are a pair of isomorphisms, and an identical argument holds for $k$ and $k^{-1}$. With composition, we obtain an isomorphism $k^{-1}\circ j$ from $(X\times_T Y)\times_U Z$ to $X\times_T (Y\times_U Z)$ which makes the entire diagram below commute:

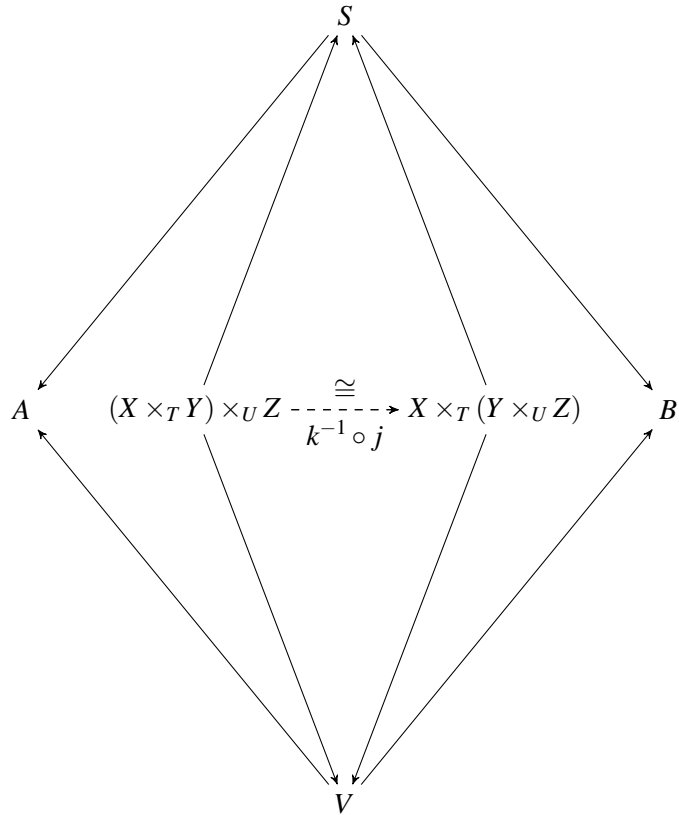
\begin{center}
\begin{tikzpicture}[->,>=stealth',auto,node distance=1.75cm,main node/.style={color=black!70!black,font=\sffamily\bfseries}]

  \node[main node] (1) {$S$};
  
              \node[color=black,font=\sffamily\bfseries] (2)[below of=1]  {};  
              \node[main node] (3)[below of=2]  {};
\node[color=black,font=\sffamily\bfseries] (4)[below of=3]  {}; 
\node[main node] (5)[below of=4]  {};
\node[color=black,font=\sffamily\bfseries] (6)[below of=5]  {}; 
\node[main node] (7)[below of=6]  {$V$};

  \node[color=black,font=\sffamily\bfseries] (8)[node distance=2cm, left of=4]  {$(X\times_T Y)\times_U Z$};
   \node[color=black,font=\sffamily\bfseries] (9)[node distance=2cm, right of=4]  {$X\times_T(Y\times_U Z)$};       
  \node[color=black,font=\sffamily\bfseries] (10)[node distance=2.3cm, left of=8]  {$A$};
   \node[color=black,font=\sffamily\bfseries] (11)[node distance=2.3cm, right of=9]  {$B$};

        \path[every node/.style={font=\sffamily\small}]
 (1)edge  node[above](25) {} (10);
   \path[every node/.style={font=\sffamily\small}](1) 
   edge node [above] (26){}(11);
  \path[every node/.style={font=\sffamily\small}]
 (7)edge  node[above](21){}  (10);
   \path[every node/.style={font=\sffamily\small}](7) 
   edge node [above] (22){}(11);
        \path[every node/.style={font=\sffamily\small}]
 (8)edge  node[above](25) {} (1);
   \path[every node/.style={font=\sffamily\small}](8) 
   edge node [above] (26){}(7);
        \path[every node/.style={font=\sffamily\small}]
 (9)edge  node[above](25) {} (1);
   \path[every node/.style={font=\sffamily\small}](9) 
   edge node [above] (26){}(7);
\draw[->,dashed] (8) to node {$\cong$} (9);
\draw[->,dashed] (8) to node [swap] {$k^{-1}\circ j$} (9);

\end{tikzpicture}
\end{center}
To verify commutativity, note that
$$(X\times_T Y)\times_U Z \rightarrow S = (X\times_T Y)\times_U Z \xrightarrow{j} L \rightarrow S =(X\times_T Y)\times_U Z \xrightarrow{k^{-1}\circ j}X\times_T (Y\times_U Z) \rightarrow S$$
using the commutativity properties of $j$ and $k^{-1}$; the bottom half of the diagram is commutative in the same way. This establishes that $(X\times_T Y)\times_U Z$ and $X\times_T (Y\times_U Z)$ are isomorphic as spans of spans, so $Z\circ (Y\circ X)=(Z\circ Y)\circ X$ - our composition rule for Span$(A,B)$ is associative. [\emph{Remark: this technique of showing two iterated pullbacks are isomorphic as spans of spans, by identifying both of them as canonically isomorphic to some limit, will be used repeatedly in this paper.}] Next we handle the issue of identity morphisms. The following diagram illustrates the identity morphism on the object $_{A}S_{B}$:

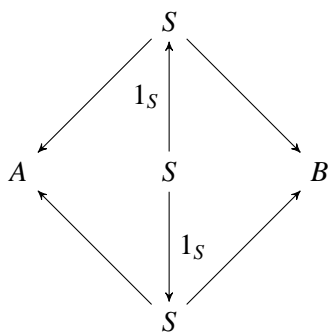
\begin{center}
\begin{tikzpicture}[->,>=stealth',node distance=2cm, auto]
 \node (A) {$A$};
 \node (S) [above of=A,right of=A] {$S$};
 \node (B) [below of=S,right of=S] {$B$};
 \node (S1) [below of=S] {$S$};
 \node (S2) [below of=S1] {$S$};
 \draw[->] (S) to node {} (A);
 \draw[->] (S) to node {} (B);
 \draw[->] (S2) to node {} (A);
 \draw[->] (S2) to node {} (B);
 \draw[->] (S1) to node {$1_S$} (S);
 \draw[->] (S1) to node {$1_S$}(S2);
\end{tikzpicture}
\end{center}
Denote this span of spans by Id$_S$. To see that this morphism behaves as an identity should, we consider these two diagrams:

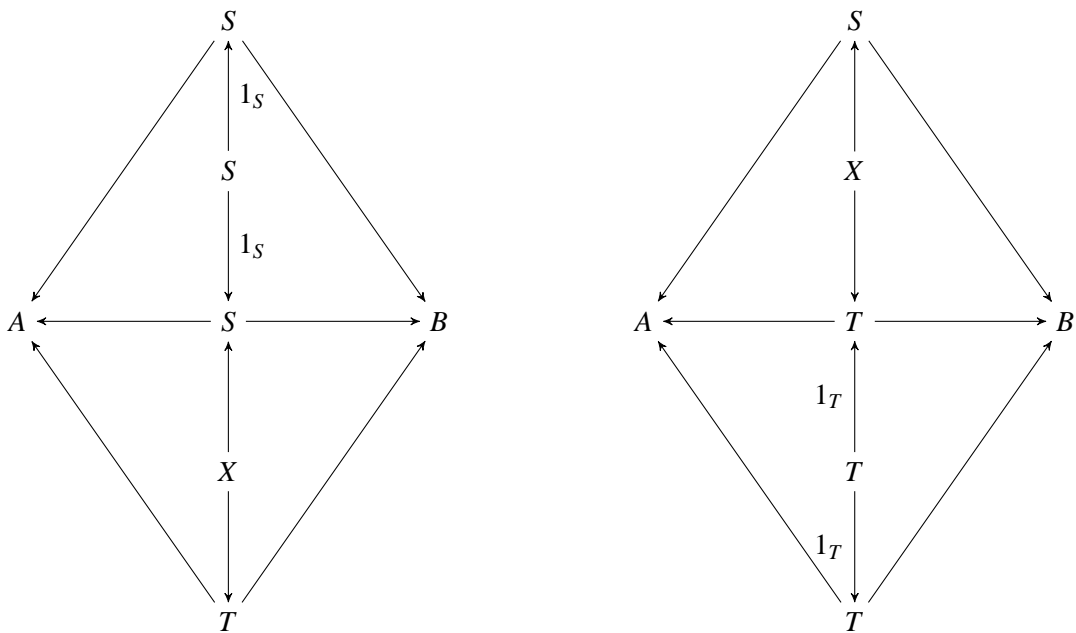
\begin{minipage}[b]{0.5\textwidth}
\centering
\begin{tikzpicture}[->,>=stealth',node distance=2cm, auto]
  \node (S) {$S$};
  \node (S1) [below of=S] {$S$};
  \node (S2) [below of=S1] {$S$};
  \node (X) [below of=S2] {$X$};
  \node (T) [below of=X] {$T$};
  \node (A) [node distance=2.8cm, left of=S2] {$A$};
  \node (B) [node distance=2.8 cm, right of=S2] {$B$};
  \draw[->] (S) to node {} (A);
  \draw[->] (S) to node {} (B);
  \draw[->] (S2) to node {}(A);
  \draw[->] (S2) to node {} (B);
  \draw[->] (T) to node {} (A);
  \draw[->] (T) to node {} (B);
  \draw[->] (S1) to node [swap] {$1_S$} (S);
  \draw[->] (S1) to node {$1_S$} (S2);
  \draw[->] (X) to node {} (S2);
  \draw[->] (X) to node {} (T);
\end{tikzpicture}

\end{minipage}
\begin{minipage}[b]{0.5\textwidth}
\centering
\begin{tikzpicture}[->,>=stealth',node distance=2cm, auto]
   \node (S) {$S$};
  \node (X) [below of=S] {$X$};
  \node (T) [below of=X] {$T$};
  \node (T1) [below of=T] {$T$};
  \node (T2) [below of=T1] {$T$};
  \node (A) [node distance=2.8cm, left of=S2] {$A$};
  \node (B) [node distance=2.8 cm, right of=S2] {$B$};
  \draw[->] (S) to node {} (A);
  \draw[->] (S) to node {} (B);
  \draw[->] (T) to node {}(A);
  \draw[->] (T) to node {} (B);
  \draw[->] (T2) to node {} (A);
  \draw[->] (T2) to node {} (B);
  \draw[->] (X) to node{} (S);
  \draw[->] (X) to node {} (T);
  \draw[->] (T1) to node {$1_T$} (T);
  \draw[->] (T1) to node [swap]{$1_T$} (T2);
\end{tikzpicture}

\end{minipage}

Recall that when one of the legs of a pullback square is an identity morphism, the resulting pullback object is canonically isomorphic to the tail of the other leg. Diagramatically:

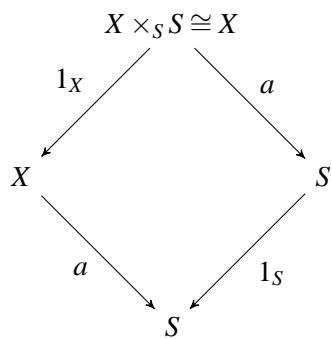
\begin{center}
\begin{tikzpicture}[->,>=stealth',node distance=2cm, auto]
 \node (S) {$S$};
 \node (X) [above of=S,left of=S] {$X$};
 \node (S1) [above of=S,right of=S] {$S$};
 \node (X1) [above of=X,right of=X] {$X\times_S S\cong X$};
 \draw[->] (X1) to node [swap]{$1_X$} (X);
 \draw[->] (X1) to node {$a$} (S1);
 \draw[->] (X) to node [swap]{$a$} (S);
 \draw[->] (S1) to node {$1_S$} (S);
\end{tikzpicture}
\end{center}
Therefore, we obtain the identities $_{S}X_{T}\circ \mbox{Id}_S=\, {_{S}X_T}=\mbox{Id}_T\circ \,_{S}X_{T}$, completing the proof that Span$(A,B)$ with isomorphism classes of spans of spans as its morphisms is a category.

\end{proof}
At this stage we introduce, for any object $A$, the functor $I_A$ from the terminal category to Span$(A,A)$, defined as follows: the action of $I_A$ on the unique object of the terminal category is the identity span $_{A}A_A$

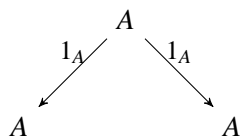
\begin{center}
\begin{tikzpicture}[->,>=stealth',auto,node distance=2cm, main node/.style={color=black!70!black,font=\sffamily}]

  \node[main node] (1) {$A$};
  \node[color=black,font=\sffamily\bfseries] (2)[below left of=1]  {$A$};
   \node[color=black,font=\sffamily\bfseries] (3)[below right of=1]  {$A$};

  \path[every node/.style={font=\sffamily\small}]
 (1)edge  node[above](6){$1_A$}  (2);
   \path[every node/.style={font=\sffamily\small}](1) 
   edge node [above] (7){$1_A$}(3);

\end{tikzpicture}
\end{center}
and the action of $I_A$ on the unique arrow of the terminal category is the identity morphism Id$_{_{A}A_A}$

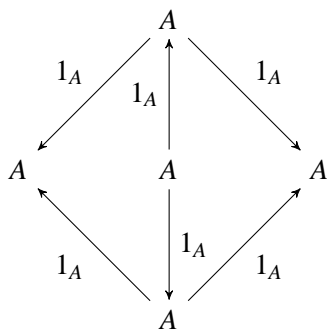
\begin{center}
\begin{tikzpicture}[->,>=stealth',node distance=2cm, auto]
 \node (A) {$A$};
 \node (A1) [right of=A,above of=A] {$A$};
 \node (A2) [right of=A] {$A$};
 \node (A3) [below of=A2] {$A$};
 \node (A4) [right of=A2] {$A$};
 \draw[->] (A1) to node [swap]{$1_A$} (A);
 \draw[->] (A1) to node {$1_A$} (A4);
 \draw[->] (A2) to node {$1_A$} (A1);
 \draw[->] (A2) to node {$1_A$} (A3);
 \draw[->] (A3) to node {$1_A$} (A);
 \draw[->] (A3) to node [swap]{$1_A$}(A4);
\end{tikzpicture}
\end{center}

This is a functor by construction, and there is nothing to prove about it for the moment, but it will appear later. We record this as
\begin{lemma} For any object $A$ of $\mathcal{C}$, there is a functor $I_A:1\rightarrow \mbox{Span}(A,A)$, where $1$ denotes the terminal category.
\end{lemma}
\begin{proof} Just given.
\end{proof}
\section{Horizontal composition \& the associator}
The next step toward showing Span$_2(\mathcal{C})$ is a bicategory is to define a composition law $$;_{A,B,C}:\mbox{Span}(A,B)\times \mbox{Span}(B,C) \rightarrow \mbox{Span}(A,C)$$ for all objects $A,B,C$ of $\mathcal{C}$. This law has to be a bifunctor. We will begin by defining the action of $;_{A,B,C}$ on objects and morphisms, and then show it is in fact a bifunctor. For an object $(_{A}S_{B},_{B}S_{C}')$ in $\mbox{Span}(A,B)\times \mbox{Span}(B,C)$, define $;_{A,B,C}(S,S')= \,_{A}(S\times_B S')_{C}$ (see the diagram below).

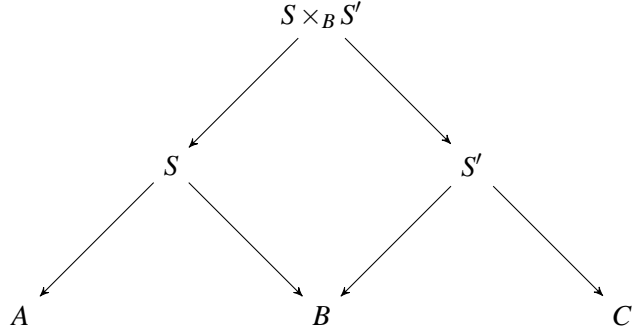
\begin{center}
\begin{tikzpicture}[->,>=stealth',node distance=2cm,auto]
 \node (A) {$A$};
 \node (S) [above of=A,right of=A] {$S$};
 \node (B) [below of=S,right of=S] {$B$};
 \node (S1) [above of=B,right of=B] {$S'$};
 \node (C) [below of=S1,right of=S1] {$C$};
 \node (SxS1) [above of=S,right of=S] {$S\times_B S'$};
 \draw[->] (S) to node {} (A);
 \draw[->] (S) to node {} (B);
 \draw[->] (S1) to node {} (B);
 \draw[->] (S1) to node {} (C);
 \draw[->] (SxS1) to node {} (S);
 \draw[->] (SxS1) to node {} (S1);
\end{tikzpicture}
\end{center}
In this definition, we stipulate that any two pullbacks of $S,S'$ over $B$ represent the same composite of $S$ and $S'$. Next we define the action of $;_{A,B,C}$ on morphisms. Given a morphism $(_{S}X_{T},_{S'}X'_{T'})$ of $\mbox{Span}(A,B)\times \mbox{Span}(B,C)$, we need $;_{A,B,C}(X,X')$ to be some representative of an isomorphism class of spans of the spans $S\times_B S'$ and $T\times_B T'$. Observe that $X\times_B X'$ is a span from $S\times_B S'$ to $T\times_B T'$, as the next diagram shows.

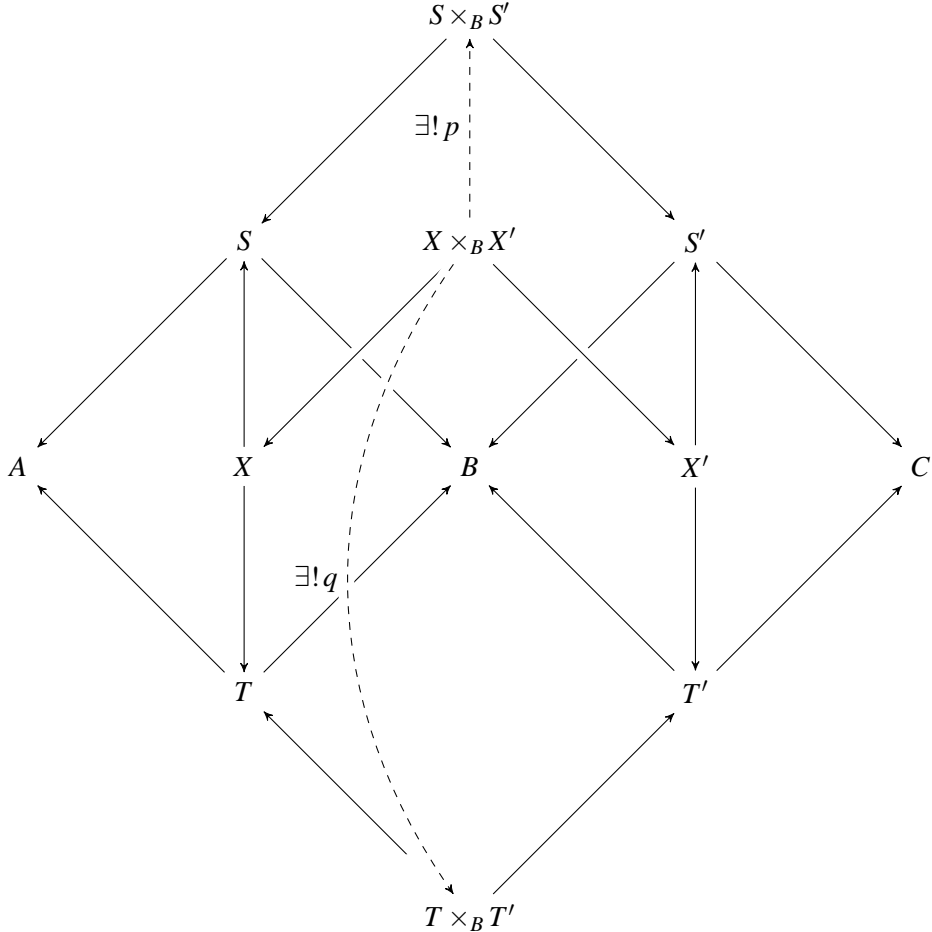
\begin{center}
\begin{tikzpicture}[->,>=stealth',node distance=3cm, auto]
 \node (A) {$A$};
 \node (S) [above of=A,right of=A] {$S$};
 \node (B) [below of=S,right of=S] {$B$};
 \node (S1) [above of=B,right of=B] {$S'$};
 \node (C) [below of=S1,right of=S1] {$C$};
 \node (T) [below of=A,right of=A] {$T$};
 \node (T1) [below of=B,right of=B] {$T'$};
 \node (X) [below of=S] {$X$};
 \node (X1) [below of=S1] {$X'$};
 \node (SxS1) [above of=S,right of=S] {$S\times_B S'$};
 \node (TxT1) [below of=T,right of=T] {$T\times_B T'$};
 \node (XxX1) [above of=B] {$X\times_B X'$};
 \draw[->] (S) to node {} (A);
 \draw[->] (S) to node {} (B);
 \draw[->] (S1) to node {} (B);
 \draw[->] (S1) to node {} (C);
 \draw[->] (SxS1) to node {} (S);
 \draw[->] (SxS1) to node {} (S1);
 \draw[->] (T) to node {} (A);
 \draw[->] (T) to node {}(B);
 \draw[->] (T1) to node {}(B);
 \draw[->] (T1) to node {} (C);
 \draw[->] (TxT1) to node{} (T);
 \draw[->] (TxT1) to node {}(T1);
 \draw[->] (X) to node {}(S);
 \draw[->] (X) to node {}(T);
 \draw[->] (X1) to node {}(S1);
 \draw[->] (X1) to node {}(T1);
 \draw[draw=white, -, line width=6pt] (XxX1) -- (X);
 \draw[draw=white, -, line width=6pt] (XxX1) -- (X1);
 \draw[->] (XxX1) to node {} (X);
 \draw[->] (XxX1) to node {}(X1);
 \draw[->,dashed] (XxX1) to node {$\exists !\,p$} (SxS1);
 \draw[->,white,line width=6pt,bend right=35] (XxX1) to node {} (TxT1);
 \draw[->,dashed,bend right=35] (XxX1) to node [swap]{$\exists !\,q$} (TxT1);

\end{tikzpicture}
\end{center}
The two unique dashed arrows in the above diagram arise because $X\times_B X'$ is a cone on the cospans $S \rightarrow B \leftarrow S'$ and $T \rightarrow B \leftarrow T'$. Because these arrows make the diagram commute by the universal properties of $S\times_B S'$ and $T\times_B T'$, we see $X\times_B X'$ is indeed a span of spans.  We let $X;X'$ denote the action of $;_{A,B,C}$ on $(X,X')$ and call it `horizontal composition of spans of spans' (as opposed to `vertical composition', the composition law within the hom-category Span$(A,B)$). Also, we call the action of $;_{A,B,C}$ on objects $(S,S')$ `composition of spans', and denote it $S;S'$. Now to see $;_{A,B,C}$ is a bifunctor, we need to confirm that it preserves identities and morphism composition.

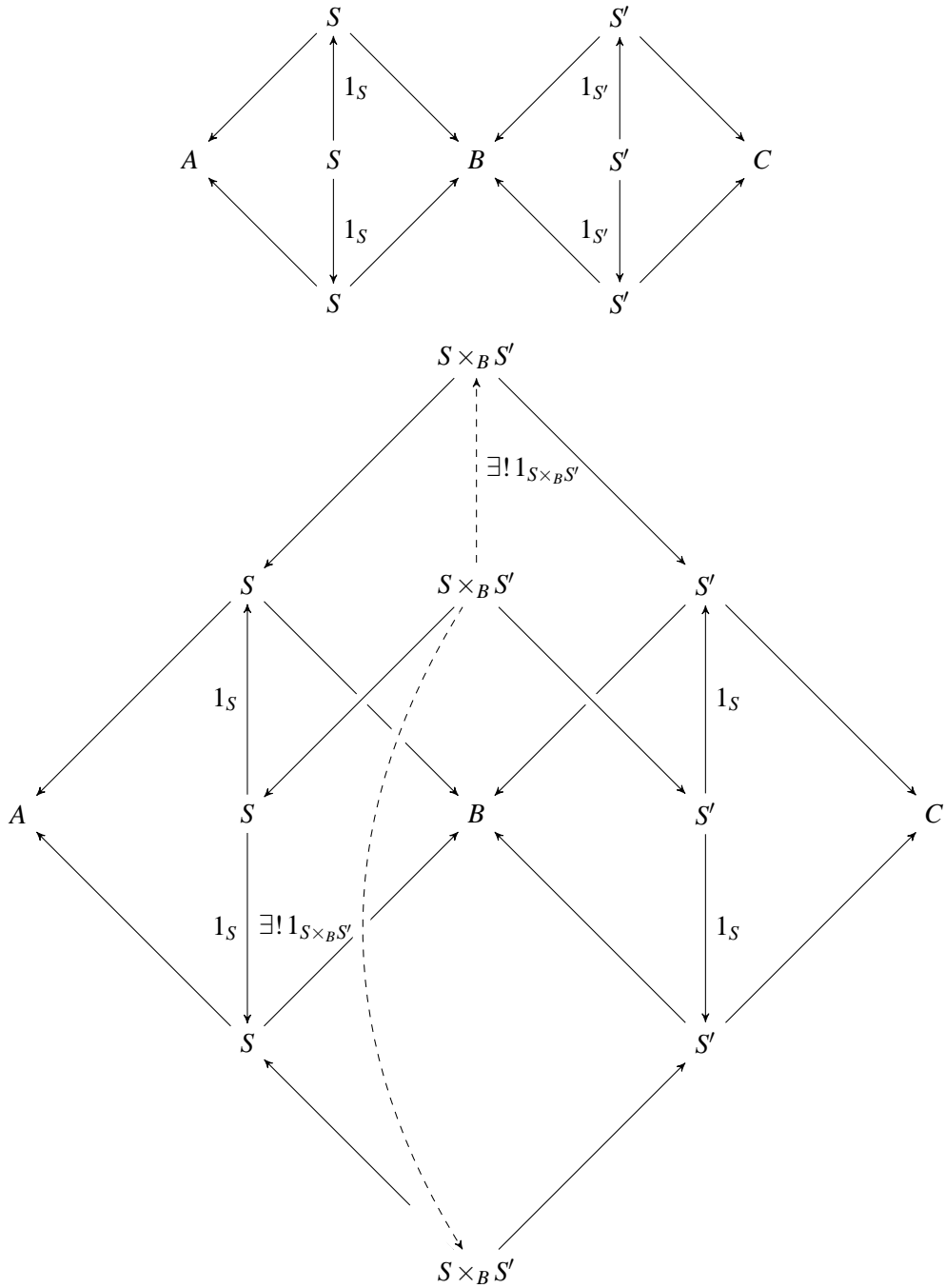
\begin{center}
\begin{tikzpicture}[->,>=stealth',node distance=2cm,auto]
 \node (S) {$S$};
 \node (S1) [below of=S] {$S$};
 \node (S2) [below of=S1] {$S$};
 \node (A) [left of=S1] {$A$}; 
 \node (B) [right of=S1] {$B$};
 \node (T1) [right of=B] {$S'$};
 \node (T) [above of=T1] {$S'$};
 \node (T2) [below of=T1] {$S'$};
 \node (C) [right of=T1] {$C$};
 \draw[->] (S) to node {}(A);
 \draw[->] (S) to node {} (B);
 \draw[->] (T) to node {} (B);
 \draw[->] (T) to node {}(C);
 \draw[->] (S2) to node {}(A);
 \draw[->] (S2) to node {}(B);
 \draw[->] (T2) to node {}(B);
 \draw[->] (T2) to node {} (C);
 \draw[->] (S1) to node[swap]{$1_S$}(S);
 \draw[->] (S1) to node{$1_S$}(S2);
 \draw[->] (T1) to node{$1_{S'}$}(T);
 \draw[->](T1) to node[swap]{$1_{S'}$}(T2);
\end{tikzpicture}
\end{center}

\begin{center}
\begin{tikzpicture}[->,>=stealth',node distance=3.2cm,auto]
 \node (S) {$S$};
 \node (S1) [below of=S] {$S$};
 \node (S2) [below of=S1] {$S$};
 \node (A) [left of=S1] {$A$}; 
 \node (B) [right of=S1] {$B$};
 \node (T1) [right of=B] {$S'$};
 \node (T) [above of=T1] {$S'$};
 \node (T2) [below of=T1] {$S'$};
 \node (C) [right of=T1] {$C$};
 \node (SxT) [above of=S,right of=S] {$S\times_B S'$};
 \node (S2xT2) [below of=S2,right of=S2] {$S\times_B S'$};
 \node (S1xT1) [above of=B] {$S\times_B S'$};
 \draw[->] (S) to node {}(A);
 \draw[->] (S) to node {} (B);
 \draw[->] (T) to node {} (B);
 \draw[->] (T) to node {}(C);
 \draw[->] (S2) to node {}(A);
 \draw[->] (S2) to node {}(B);
 \draw[->] (T2) to node {}(B);
 \draw[->] (T2) to node {} (C);
 \draw[->] (S1) to node{$1_S$}(S);
 \draw[->] (S1) to node[swap]{$1_S$}(S2);
 \draw[->] (T1) to node[swap]{$1_S$}(T);
 \draw[->](T1) to node{$1_S$}(T2);
 \draw[->](SxT) to node {}(S);
 \draw[->] (SxT) to node {}(T);
 \draw[->](S2xT2) to node {}(S2);
 \draw[->] (S2xT2) to node {}(T2);
 \draw[->,dashed] (S1xT1) to node[swap]{$\exists !\,1_{S\times_B S'}$} (SxT);
 \draw[->,white,line width=7pt, bend right=32] (S1xT1) to node {}(S2xT2);
 \draw[->,dashed,bend right=32] (S1xT1) to node [swap] {$\exists !\,1_{S\times_B S'}$} (S2xT2);
 \draw[->,white,line width=6pt] (S1xT1) to node {}(S1);
 \draw[->,white,line width=6pt] (S1xT1) to node {} (T1);
 \draw[->](S1xT1) to node{}(S1);
 \draw[->](S1xT1) to node{}(T1);
\end{tikzpicture}
\end{center}
The image of the identity morphism $(\mbox{Id}_S,\mbox{Id}_{S'})$ through $;_{A,B,C}$ is indeed the identity morphism from $S\times_B S'$ to itself, as the above two diagrams illustrate. Next we need to show preservation of morphism composition through $;_{A,B,C}$, also called the interchange law of a bicategory. We begin with two pairs of composable morphisms:

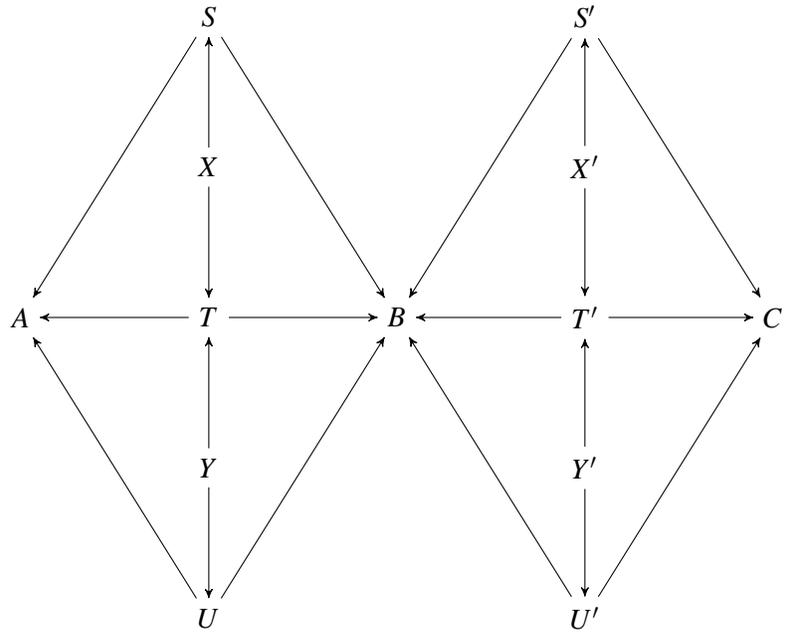
\begin{center}
\begin{tikzpicture}[->,>=stealth',node distance=2cm,auto]
 \node (S) {$S$};
 \node (X) [below of=S] {$X$};
 \node (T) [below of=X] {$T$};
 \node (Y) [below of=T] {$Y$};
 \node (U) [below of=Y] {$U$};
 \node (A) [node distance=2.5cm, left of=T] {$A$};
 \node (B)[node distance=2.5cm, right of=T] {$B$};
 \node (T1) [node distance=2.5cm, right of=B] {$T'$};
 \node (X1) [above of=T1] {$X'$};
 \node (S1) [above of=X1] {$S'$};
 \node (Y1) [below of=T1] {$Y'$};
 \node (U1) [below of=Y1] {$U'$};
 \node (C) [node distance=2.5cm, right of=T1] {$C$};
 \draw[->] (S) to node {}(A);
\draw[->] (S) to node {}(B);
\draw[->] (T) to node {}(A);
 \draw[->] (T) to node {}(B);
 \draw[->](U) to node {}(A);
 \draw[->](U) to node {}(B);
 \draw[->](S1) to node {}(B);
 \draw[->](S1) to node {}(C);
 \draw[->](T1) to node {} (B);
 \draw[->](T1) to node {}(C);
 \draw[->] (U1) to node {} (B);
 \draw[->] (U1) to node {}(C);
 \draw[->] (X) to node {} (S);
 \draw[->] (X) to node {}(T);
 \draw[->] (Y) to node {}(T);
 \draw[->](Y) to node {}(U);
 \draw[->] (X1) to node {}(S1);
 \draw[->] (X1) to node {} (T1);
 \draw[->] (Y1) to node {} (T1);
 \draw[->] (Y1) to node {} (U1);
\end{tikzpicture}
\end{center}
The result of composing these morphisms before applying $;_{A,B,C}$ is

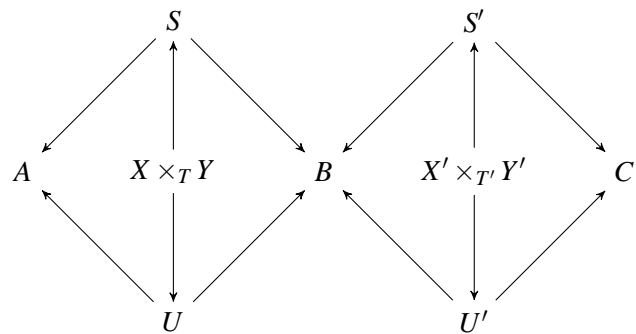
\begin{center}
\begin{tikzpicture}[->,>=stealth',node distance=2cm,auto]
 \node (S) {$S$};
 \node (XxY) [below of=S] {$X\times_T Y$};
 \node (U) [below of=XxY] {$U$};
 \node (A) [left of=XxY] {$A$};
 \node (B) [right of=XxY] {$B$};
 \node (X1xY1) [right of=B] {$X'\times_{T'} Y'$};
 \node (S1) [above of=X1xY1] {$S'$};
 \node (U1) [below of=X1xY1] {$U'$};
 \node (C) [right of=X1xY1] {$C$};
 \draw[->](S) to node {}(A);
 \draw[->](S) to node {}(B);
 \draw[->](XxY) to node {}(S);
 \draw[->](XxY) to node {}(U);
 \draw[->](U) to node {}(A);
 \draw[->](U) to node {}(B);
 \draw[->](S1) to node {}(B);
 \draw[->](S1) to node {}(C);
 \draw[->](X1xY1) to node {}(S1);
 \draw[->](X1xY1) to node {}(U1);
 \draw[->](U1) to node {}(B);
 \draw[->](U1) to node {}(C);
\end{tikzpicture}
\end{center}
and now applying $;_{A,B,C}$ to these morphisms yields the span of spans

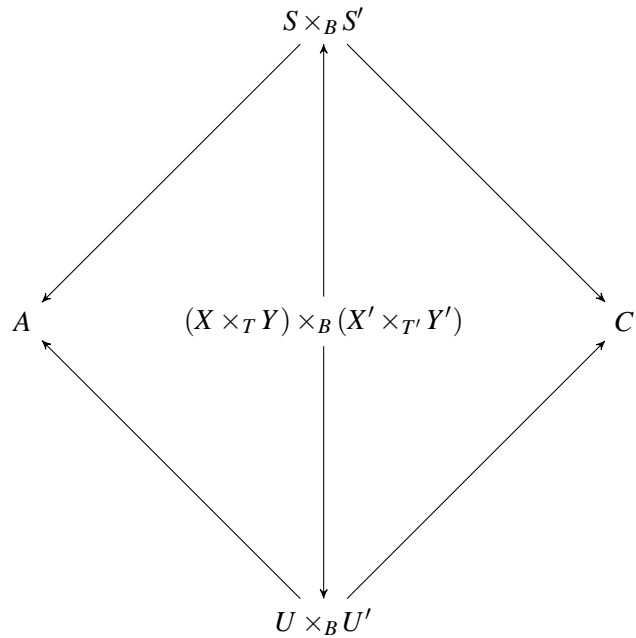
\begin{center}
\begin{tikzpicture}[->,>=stealth',node distance=2cm,auto]
 \node (SxS1) {$S\times_B S'$};
 \node (XYX1Y1) [node distance=4cm,below of=SxS1] {$(X\times_T Y)\times_B (X'\times_{T'} Y')$};
 \node (UxU1) [node distance=4cm,below of=XYX1Y1] {$U\times_B U'$};
 \node (A) [node distance=4cm,left of=XYX1Y1] {$A$};
 \node (C) [node distance=4cm, right of=XYX1Y1] {$C$};
 \draw[->] (SxS1) to node {} (A);
 \draw[->] (SxS1) to node {}(C);
 \draw[->](XYX1Y1) to node {}(SxS1);
 \draw[->](XYX1Y1) to node {}(UxU1);
 \draw[->](UxU1) to node {}(A); 
 \draw[->] (UxU1) to node {}(C);
\end{tikzpicture}
\end{center}
On the other hand, first applying $;_{A,B,C}$ to the top pair of morphisms and to the bottom pair of morphisms yields

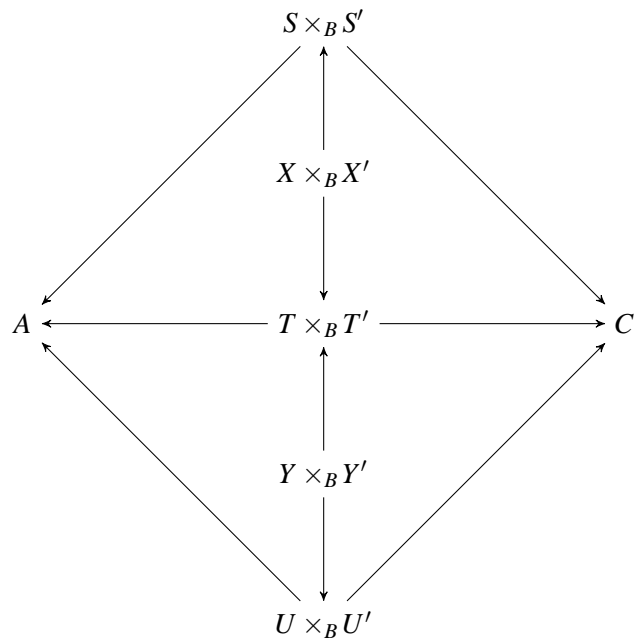
\begin{center}
\begin{tikzpicture}[->,>=stealth',node distance=2cm,auto]
 \node (SxS1) {$S\times_B S'$};
 \node (XxX1) [below of=SxS1] {$X\times_B X'$};
 \node (TxT1) [below of=XxX1] {$T\times_B T'$};
 \node (YxY1) [below of=TxT1] {$Y\times_B Y'$};
 \node (UxU1) [below of=YxY1] {$U\times_B U'$};
 \node (A) [node distance=4cm,left of=TxT1] {$A$};
 \node (C) [node distance=4cm, right of=TxT1] {$C$};
 \draw[->] (SxS1) to node {}(A);
 \draw[->] (SxS1) to node {}(C);
 \draw[->] (TxT1) to node {}(A);
 \draw[->] (TxT1) to node {}(C);
 \draw[->] (UxU1) to node {}(A);
 \draw[->] (UxU1) to node {}(C);
 \draw[->] (XxX1) to node {}(SxS1);
 \draw[->] (XxX1) to node {}(TxT1);
 \draw[->] (YxY1) to node {}(TxT1);
 \draw[->] (YxY1) to node {}(UxU1);
\end{tikzpicture}
\end{center}
and now composing these morphisms results in

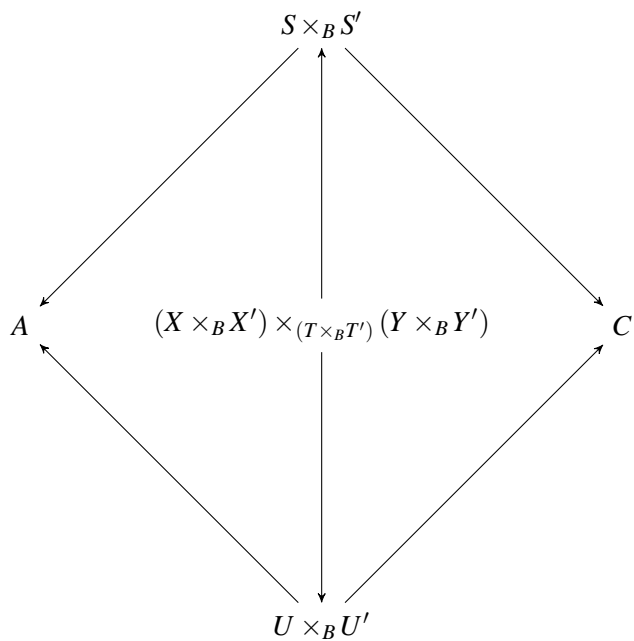
\begin{center}
\begin{tikzpicture}[->,>=stealth',node distance=2cm,auto]
 \node (SxS1) {$S\times_B S'$};
 \node (XX1YY1) [node distance=4cm,below of=SxS1] {$(X\times_B X')\times_{(T \times_B T')} (Y\times_{B} Y')$};
 \node (UxU1) [node distance=4cm,below of=XYX1Y1] {$U\times_B U'$};
 \node (A) [node distance=4cm,left of=XYX1Y1] {$A$};
 \node (C) [node distance=4cm, right of=XYX1Y1] {$C$};
 \draw[->] (SxS1) to node {} (A);
 \draw[->] (SxS1) to node {}(C);
 \draw[->](XX1YY1) to node {}(SxS1);
 \draw[->](XX1YY1) to node {}(UxU1);
 \draw[->](UxU1) to node {}(A); 
 \draw[->] (UxU1) to node {}(C);
\end{tikzpicture}
\end{center}
 For functoriality we need to see that $(X\times_T Y)\times_B (X'\times_{T'} Y')$ and $(X\times_B X')\times_{(T \times_B T')} (Y\times_{B} Y')$ are isomorphic as spans of spans. This will be accomplished similarly to how we showed the composition law in Span$(A,B)$ is associative - by showing each object in question is canonically isomorphic to the limit of a certain subdiagram. But first note that $(X\times_B X')\times_{(T\times_B T')} (Y\times_B Y')$ is exactly the same as $(X\times_B X')\times_B (Y\times_B Y')$ (since $B$ is what $T\times_B T'$ is taken over), so we will write this pullback with the latter simplified notation from now on. Consider the subdiagram

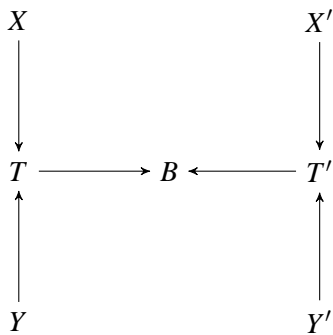
\begin{center}
\begin{tikzpicture}[->,>=stealth',node distance=2cm,auto]
 \node (X) {$X$};
 \node (T) [below of=X] {$T$};
 \node (Y) [below of=T] {$Y$};
 \node (B) [right of=T] {$B$};
 \node (T1) [right of=B] {$T'$};
 \node (X1) [above of=T1] {$X'$};
 \node (Y1) [below of=T1] {$Y'$};
 \draw[->] (X) to node {}(T);
 \draw[->](Y) to node {}(T);
 \draw[->](T) to node {}(B);
 \draw[->](T1) to node {}(B);
 \draw[->](X1) to node {}(T1);
 \draw[->](Y1) to node {}(T1);
\end{tikzpicture}
\end{center}

Letting $L$ now stand for the limit of this diagram, we have the following:

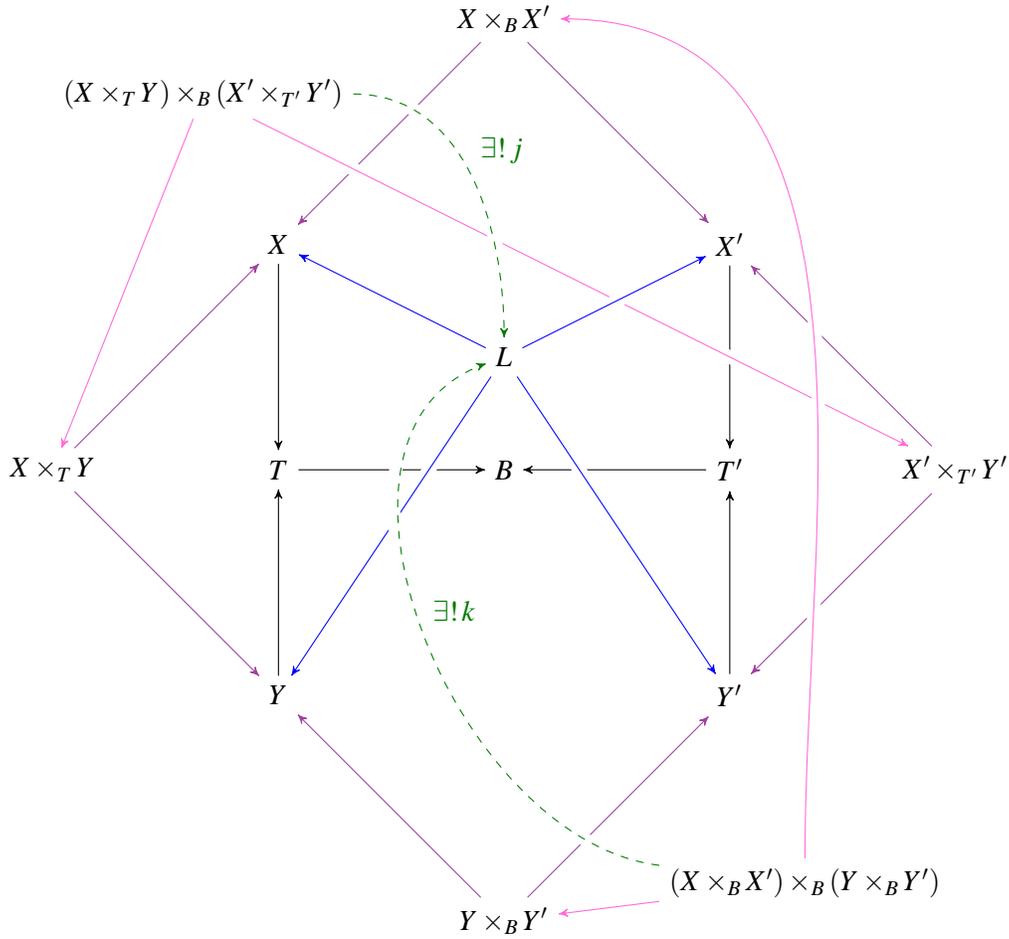
\begin{center}
\begin{tikzpicture}[->,>=stealth',node distance=3cm,auto]
 \node (X) {$X$};
 \node (T) [below of=X] {$T$};
 \node (Y) [below of=T] {$Y$};
 \node (B) [right of=T] {$B$};
 \node (T1) [right of=B] {$T'$};
 \node (X1) [above of=T1] {$X'$};
 \node (Y1) [below of=T1] {$Y'$};
 \node (XxX1) [node distance=6cm,above of=B] {$X\times_B X'$};
 \node (YxY1) [node distance=6cm, below of=B] {$Y\times_B Y'$};
 \node (L) [node distance=1.5cm,above of=B] {$L$};
 \node (XxY) [left of=T] {$X\times_T Y$};
 \node (X1xY1) [right of=T1] {$X'\times_{T'} Y'$};
 \node (XYxX1Y1) [node distance=1.5cm,above of=X,left of=X,xshift=.5cm,yshift=.5cm] {$(X\times_T Y)\times_B (X'\times_{T'} Y')$};
 \node (XX1xYY1) [node distance=1.5cm,below of=Y1,right of=Y1,xshift=-.5cm,yshift=-1cm] {$(X\times_B X')\times_{B}(Y\times_B Y')$};
 \draw[->] (X) to node {}(T);
 \draw[->](Y) to node {}(T);
 \draw[->](T) to node {}(B);
 \draw[->](T1) to node {}(B);
 \draw[->](X1) to node {}(T1);
 \draw[->](Y1) to node {}(T1);
 \draw[->,blue] (L) to node {}(X);
 \draw[->,violet!75!white!] (XxX1) to node {}(X);
 \draw[->,violet!75!white!] (XxX1) to node {}(X1);
 \draw[->,white,line width=5pt] (L) to node {}(Y);
 \draw[->,blue] (L) to node {}(Y);
\draw[->,white,line width=5pt] (L) to node {}(Y1);
 \draw[->,blue] (L) to node {}(Y1);
 \draw[->,violet!75!white!] (YxY1) to node{}(Y);
 \draw[->,violet!75!white!] (YxY1) to node{}(Y1);
 \draw[->,violet!75!white!] (XxY) to node{}(X);
\draw[->,violet!75!white!] (XxY) to node{}(Y);
 \draw[->,violet!75!white!](X1xY1) to node{}(X1);
\draw[->,violet!75!white!](X1xY1) to node{}(Y1);
\draw[->,pink!55!magenta](XYxX1Y1) to node{}(XxY);
\draw[->,white,line width=5pt] (XYxX1Y1) to node{}(X1xY1);
\draw[->,pink!55!magenta] (XYxX1Y1) to node{}(X1xY1);
\draw[->,white,line width=5pt] (L) to node {}(X1);
\draw[->,blue] (L) to node {}(X1);
\draw[->,white,line width=5pt,out=90,in=360](XX1xYY1) to node{}(XxX1);
\draw[->,out=90,in=360,pink!55!magenta](XX1xYY1) to node{}(XxX1);
\draw[->,pink!55!magenta](XX1xYY1) to node{}(YxY1);
\draw[->,white,line width=5pt,out=0,in=90] (XYxX1Y1) to node {} (L);
\draw[->,dashed,green!48!black,out=0,in=90] (XYxX1Y1) to node {$\exists !\,j$} (L);
\draw[->,white,line width=5pt,out=173,in=200] (XX1xYY1) to node {} (L);
\draw[->,dashed,green!48!black,out=173,in=200] (XX1xYY1) to node [swap]{$\exists !\,k$} (L);
\end{tikzpicture}
\end{center}

To aid the eye, arrows from the limit are in blue, arrows arising from pullbacks are in purple and pink, and unique arrows to the limit are in green. It is easy to see that $(X\times_T Y)\times_B (X' \times_{T'} Y')$ and $(X\times_B X')\times_B (Y\times_B Y')$ are both cones on the base diagram, hence the existence of arrows $j$ and $k$ uniquely making the diagram commute. Next, as in the proof of associativity in Lemma 1, we will show the existence of a unique arrow $k^{-1}$ that is an inverse of $k$ and also makes the diagram commute, so that $k^{-1}\circ j$ is an isomorphism of spans of spans (the inverse $j^{-1}$ also exists, but we will not need to use it here). Note that $L$ is a cone on the cospans determined by $X\times_B X'$ and $Y \times_B Y'$, since it is the limit of the larger diagram:

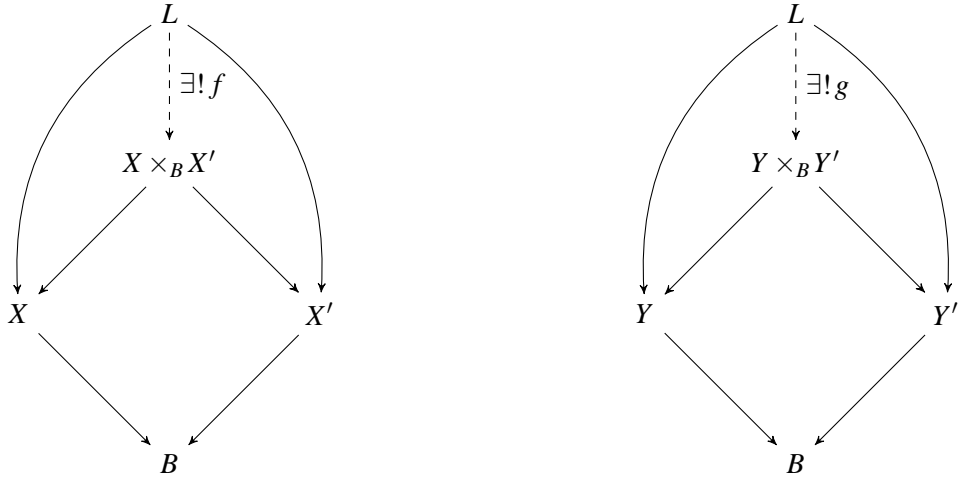
\begin{minipage}[b]{0.5\textwidth}
\centering
\begin{tikzpicture}[->,>=stealth',node distance=2cm, auto]
  \node (B) {$B$};
  \node (X) [left of=B,above of=B] {$X$};
  \node (X1) [right of=B,above of=B] {$X'$};
  \node (XxX1) [right of=X,above of=X] {$X\times_B X'$};
  \node (L) [above of=XxX1] {$L$};
  \draw[->] (X) to node {} (B);
  \draw[->] (X1) to node {} (B);
  \draw[->] (XxX1) to node {} (X);
  \draw[->] (XxX1) to node {} (X1);
 \draw[->,bend right] (L) to node {} (X);
  \draw[->,bend left] (L) to node {} (X1);
  \draw[->,dashed] (L) to node {$\exists !\,f$} (XxX1);
\end{tikzpicture}

\end{minipage}
\begin{minipage}[b]{0.5\textwidth}
\centering
\begin{tikzpicture}[->,>=stealth',node distance=2cm, auto]
  \node (B) {$B$};
  \node (Y) [left of=B,above of=B] {$Y$};
  \node (Y1) [right of=B,above of=B] {$Y'$};
  \node (YxY1) [right of=Y,above of=Y] {$Y\times_B Y'$};
  \node (L) [above of=YxY1] {$L$};
  \draw[->] (Y) to node {} (B);
  \draw[->] (Y1) to node {} (B);
  \draw[->] (YxY1) to node {} (Y);
  \draw[->] (YxY1) to node {} (Y1);
 \draw[->,bend right] (L) to node {} (Y);
  \draw[->,bend left] (L) to node {} (Y1);
  \draw[->,dashed] (L) to node {$\exists !\,g$} (YxY1);
\end{tikzpicture}

\end{minipage}

The intermediate arrows $f,g$ make $L$ a cone on the cospan determined by $(X\times_B X')\times_B (Y\times_B Y')$:

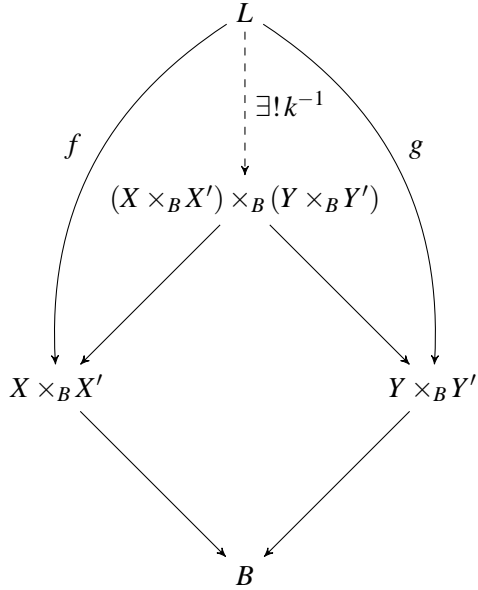
\begin{center}
\begin{tikzpicture}[->,>=stealth',node distance=2.5cm, auto]
  \node (B) {$B$};
  \node (XxX1) [left of=B,above of=B] {$X\times_B X'$};
  \node (YxY1) [right of=B,above of=B] {$Y\times_B Y'$};
  \node (XX1xYY1) [right of=XxX1,above of=XxX1] {$(X\times_B X')\times_B (Y\times_B Y')$};
  \node (L) [above of=XX1xYY1] {$L$};
  \draw[->] (XxX1) to node {} (B);
  \draw[->] (YxY1) to node {} (B);
  \draw[->]  (XX1xYY1) to node {} (XxX1);
  \draw[->]  (XX1xYY1) to node {} (YxY1);
 \draw[->,bend right] (L) to node [swap] {$f$} (XxX1);
  \draw[->,bend left] (L) to node {$g$} (YxY1);
  \draw[->,dashed] (L) to node {$\exists !\,k^{-1}$} (XX1xYY1);
\end{tikzpicture}
\end{center}

To see that $L$ really is a cone on the above cospan, note the following equality of paths of arrows:
$$L\xrightarrow{f} X\times_B X' \rightarrow B = L\rightarrow X \rightarrow B = L\rightarrow Y \rightarrow B = L \xrightarrow{g} Y\times_B Y' \rightarrow B$$
using the commuting properties of $f$ and $g$, and the fact that $L$ is a limit. Now, we omit the proof that $k^{-1}\circ k=1_{(X\times_B X')\times_B (Y\times_B Y')}$ and $k\circ k^{-1}=1_L$ since it is formally identical to the prior proof in Lemma 1 that $j$ and $j^{-1}$, or $k$ and $k^{-1}$, are inverses in that context; one just notes that the identity arrow and the composite arrow both satisfy the role of a universal arrow, as in the standard proof that any two limit objects of a given diagram are canonically isomorphic, and so those arrows are equal. Further, the proof that $k^{-1}\circ j$ is an isomorphism of spans of spans is also formally the same as before; the relevant equality of paths is given via
$$(X\times_T Y)\times_B (X'\times_{T'} Y') \rightarrow S\times_B S'=(X\times_T Y)\times_B (X'\times_{T'} Y') \xrightarrow{j} L \xrightarrow{f} X\times_B X' \rightarrow S\times_B S'$$
$$= (X\times_T Y)\times_B (X'\times_{T'} Y') \xrightarrow{k^{-1}\circ j} (X\times_B X')\times_B (Y\times_B Y')\rightarrow X\times_B X' \rightarrow S\times S'$$
and similarly 
$$(X\times_T Y)\times_B (X'\times_{T'} Y') \rightarrow U\times_B U' = (X\times_T Y)\times_B (X'\times_{T'} Y') \xrightarrow{j} L \xrightarrow{g} Y\times_B Y' \rightarrow U\times_B U'$$
$$= (X\times_T Y)\times_B (X'\times_{T'} Y') \xrightarrow{k^{-1}\circ j} (X\times_B X')\times_B(Y\times_B Y') \rightarrow Y\times_B Y' \rightarrow U\times U'.$$
We have thus established the interchange law,
$$(Y\circ X);(Y'\circ X')=(Y;Y')\circ (X;X').$$
We summarize the above results as

\begin{lemma} There is a bifunctor $;_{A,B,C}:\mbox{Span}(A,B)\times \mbox{Span}(B,C)\rightarrow \mbox{Span}(A,C)$ whose action on objects and morphisms is given by appropriate pullbacks.
\end{lemma}
\begin{proof} See above, starting at the bottom of page 10 and ending with the interchange law on this page.
\end{proof}
The next major feature of any bicategory is a natural isomorphism called the associator,
$$a_{A,B,C,D}:\;\, ;_{A,B,D}\circ(\mbox{Id}\times ;_{B,C,D}) \overset{\cong}{\Longrightarrow} \; ;_{A,C,D}\circ (;_{A,B,C}\times \mbox{Id})$$
which is required to obey a coherence law called the pentagon identity. Here Id stands for the identity functor on any given category, $\times$ stands for the product of two functors, and $\circ$ stands for composition of functors. For brevity we let $F:=\;\, ;_{A,B,D}\circ(\mbox{Id}\times ;_{B,C,D})$ and $G:=\;\, ;_{A,C,D}\circ (;_{A,B,C}\times \mbox{Id})$. Observe that $F,G$ are functors $\mbox{Span}(A,B)\times \mbox{Span}(B,C) \times \mbox{Span}(C,D)\rightarrow \mbox{Span}(A,D)$, given by
$$F(_{A}S_B,\,_{B}S'_C,\,_{C}S''_D)=S\times_B(S'\times_C S''),\;\;\;\;\;[\mbox{action on objects}]$$
$$F(_{S}X_T,\,_{S'}X'_{T'},\,_{S''}X''_{T''})=X\times_B(X'\times_C X''),\;\;\;\;\;[\mbox{action on morphisms}]$$
$$G(_{A}S_B,\,_{B}S'_C,\,_{C}S''_D)=(S\times_B S')\times_C S'',$$
$$G(_{S}X_{T},\,_{S'}X'_{T'},\,_{S''}X''_{T''})=(X\times_B X')\times_C X''.$$
In words, $F$ parenthesizes a triple (of composable spans or horizontally composable spans of spans) on the right, and $G$ parenthesizes on the left. Notice that these pullbacks are of the same form as those discussed above in proving the associativity of (vertical) composition for Lemma 1. There, we considered the diagram

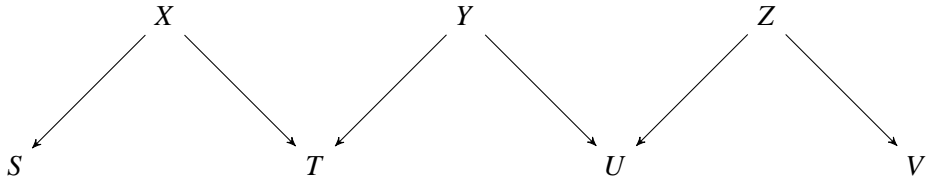
\begin{center}
\begin{tikzpicture}[->,>=stealth',node distance=2cm, auto]
 \node (S) {$S$};
 \node (X) [above of=S,right of=S] {$X$};
 \node (T) [below of=X,right of=X] {$T$};
 \node (Y) [above of=T,right of=T] {$Y$};
 \node (U) [below of=Y,right of=Y] {$U$};
 \node (Z) [above of=U,right of=U] {$Z$};
 \node (V) [below of=Z,right of=Z] {$V$};
 \draw[->] (X) to node {}(S);
 \draw[->] (X) to node {}(T);
 \draw[->] (Y) to node {}(T);
 \draw[->] (Y) to node {}(U);
 \draw[->] (Z) to node {}(U);
 \draw[->] (Z) to node {}(V);
\end{tikzpicture}
\end{center}
where $S,T,U,$ and $V$ at the bottom are all objects in Span$(A,B)$, and saw that $(X\times_T Y)\times_U Z$ and $X\times_T (Y\times_U Z)$ are isomorphic as spans of spans. The part of the proof that showed the existence of the isomorphism made no reference to the nature of the objects at the bottom of the diagram, and so if we instead consider

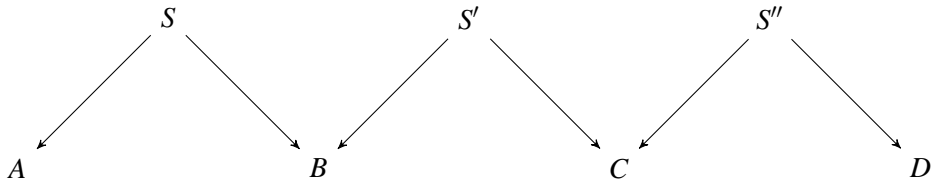
\begin{center}
\begin{tikzpicture}[->,>=stealth',node distance=2cm, auto]
 \node (A) {$A$};
 \node (S) [above of=A,right of=A] {$S$};
 \node (B) [below of=S,right of=S] {$B$};
 \node (S1) [above of=B,right of=B] {$S'$};
 \node (C) [below of=S1,right of=S1] {$C$};
 \node (S2) [above of=C,right of=C] {$S''$};
 \node (D) [below of=S2,right of=S2] {$D$};
 \draw[->] (S) to node {}(A);
 \draw[->] (S) to node {}(B);
 \draw[->] (S1) to node {}(B);
 \draw[->] (S1) to node {}(C);
 \draw[->] (S2) to node {}(C);
 \draw[->] (S2) to node {}(D);
\end{tikzpicture}
\end{center}
where $A,B,C,$ and $D$ are simply objects of $\mathcal{C}$, we still get that $(S\times_B S')\times_C S''$ and $S\times_B (S'\times_C S'')$ are isomorphic, but only as spans. If we again use $L$ to denote the limit of the base diagram, and $j$ and $k$ for the unique maps that arise from the pullbacks as seen above, we then have

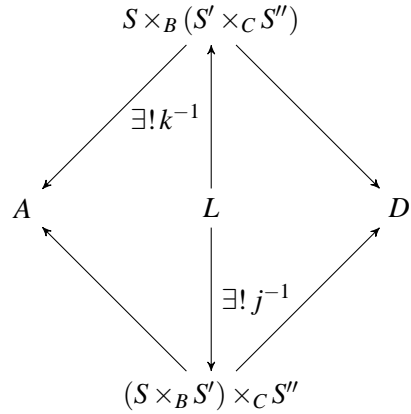
\begin{center}
\begin{tikzpicture}[->,>=stealth',node distance=2.5cm, auto]
 \node (A) {$A$};
 \node (SxS1S2) [above of=A,right of=A] {$S\times_B (S'\times_C S'')$};
 \node (L) [right of=A] {$L$};
 \node (D) [right of=L] {$D$};
 \node (SS1xS2) [below of=L] {$(S\times_B S')\times_C S''$};
 \draw[->] (SxS1S2) to node {}(A);
 \draw[->] (SxS1S2) to node {}(D);
 \draw[->] (SS1xS2) to node {}(A);
 \draw[->] (SS1xS2) to node {}(D);
\draw[->] (L) to node {$\exists !\,k^{-1}$} (SxS1S2);
\draw[->] (L) to node {$\exists !\,j^{-1}$} (SS1xS2);
\end{tikzpicture}
\end{center}
Observe that in this diagram, $L$ is a span of spans, i.e. it makes the diagram commute, due to the unique commutative properties of $j^{-1}$ and $k^{-1}$. It is also an invertible span of spans with respect to vertical composition:

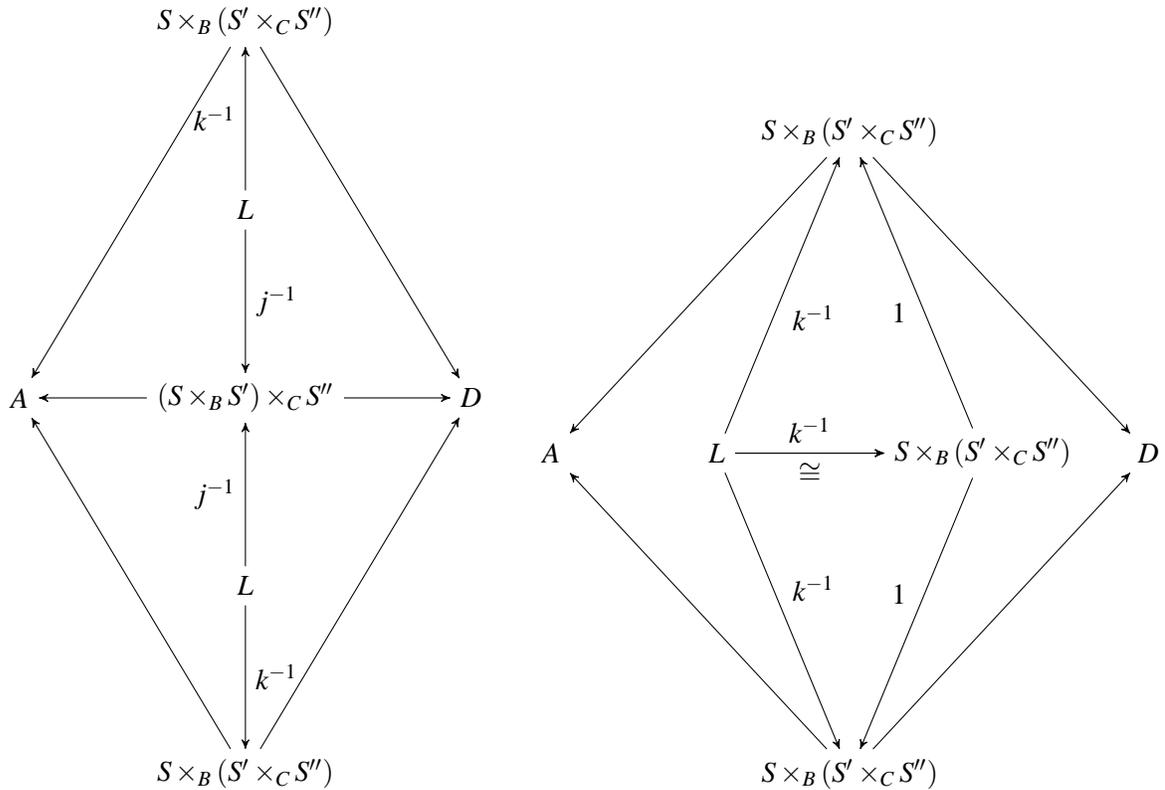
\begin{minipage}[b]{0.45\textwidth}
\centering
\begin{tikzpicture}[->,>=stealth',node distance=2.5cm, auto]
\node (A) {$A$};
\node (SS1xS2) [right of=A,xshift=.5cm] {$(S\times_B S')\times_C S''$};
\node (L1) [above of=SS1xS2] {$L$};
\node (SxS1S2) [above of=L1] {$S\times_B (S'\times_C S'')$};
\node (L2) [below of=SS1xS2] {$L$};
\node (S3) [below of=L2] {$S\times_B(S'\times_C S'')$};
\node (D) [right of=SS1xS2,xshift=.5cm] {$D$};
\draw[->] (SS1xS2) to node {} (A);
\draw[->] (SS1xS2) to node {}(D);
\draw[->] (SxS1S2) to node {}(A);
\draw[->] (SxS1S2) to node {}(D);
\draw[->] (S3) to node {}(A);
\draw[->] (S3) to node {}(D);
\draw[->] (L1) to node {$k^{-1}$}(SxS1S2);
\draw[->] (L1) to node {$j^{-1}$}(SS1xS2);
\draw[->] (L2) to node {$j^{-1}$}(SS1xS2);
\draw[->] (L2) to node {$k^{-1}$}(S3);
\end{tikzpicture}

\end{minipage}
\begin{minipage}[b]{0.45\textwidth}
\centering
\begin{tikzpicture}[->,>=stealth',node distance=2.5cm, auto]
\node (1) {};
\node (SxS1S2) [above of=1] {$S\times_B (S'\times_C S'')$};
\node (L) [below left of=1] {$L$};
\node (SxS1S2b) [below right of=1] {$S\times_B (S'\times_C S'')$};
\node (2) [below right of=L] {};
\node (SxS1S2c) [below of=2] {$S\times_B (S'\times_C S'')$};
\node (A) [left of=L,xshift=.3cm] {$A$};
\node (D) [right of=SxS1S2b,xshift=-.3cm] {$D$};
\draw[->] (SxS1S2) to node {}(A);
\draw[->] (SxS1S2) to node {}(D);
\draw[->] (SxS1S2c) to node {}(A);
\draw[->] (SxS1S2c) to node {}(D);
\draw[->] (L) to node {$k^{-1}$}(SxS1S2c);
\draw[->] (L) to node [swap]{$k^{-1}$}(SxS1S2);
\draw[->] (SxS1S2b) to node [swap]{$1$}(SxS1S2c);
\draw[->] (SxS1S2b) to node {$1$}(SxS1S2);
\draw[->] (L) to node {$k^{-1}$}(SxS1S2b);
\draw[->] (L) to node [swap]{$\cong$}(SxS1S2b);

\end{tikzpicture}

\end{minipage}\\
The reverse composition is seen in the same way to be isomorphic to the identity span of spans Id$_{(S\times_B S')\times_C S''}$. Therefore, we define the associator $a_{A,B,C,D}$ componentwise as follows: the component at the object $(_{A}S_B,\,_{B}S'_C,\,_{C}S''_D)$ is the invertible span of spans: 
\begin{center}
\begin{tikzpicture}[->,>=stealth',node distance=2.5cm, auto]
 \node (A) {$A$};
 \node (SxS1S2) [above of=A,right of=A] {$S\times_B (S'\times_C S'')$};
 \node (L) [right of=A] {$L$};
 \node (D) [right of=L] {$D$};
 \node (SS1xS2) [below of=L] {$(S\times_B S')\times_C S''$};
 \draw[->] (SxS1S2) to node {}(A);
 \draw[->] (SxS1S2) to node {}(D);
 \draw[->] (SS1xS2) to node {}(A);
 \draw[->] (SS1xS2) to node {}(D);
\draw[->] (L) to node {$\exists !\,k^{-1}$} (SxS1S2);
\draw[->] (L) to node {$\exists !\,j^{-1}$} (SS1xS2);
\end{tikzpicture}
\end{center}
To prove that this associator is in fact a natural isomorphism from $F$ to $G$, it will be convenient to denote the limit used in the above diagram as $L_{S,S',S''}$ instead of just $L$, and similarly for the arrows from $L$ to the pullbacks. The naturality condition to show is then the commutativity of the following square (treating each edge of the square as a single morphism that points either right or down, not a pair of arrows):

\begin{center}
\begin{tikzpicture}[->,>=stealth',node distance=4cm, auto]
\node (SxSS) {$S\times_B (S'\times_C S'')$};
\node (XxXX) [right of=SxSS] {$X\times_B (X'\times_C X'')$};
\node (LS) [below of=SxSS] {$L_{S,S',S''}$};
\node (SSxS) [below of=LS] {$(S\times_B S')\times_C S''$};
\node (TxTT) [right of=XxXX] {$T\times_B (T'\times_C T'')$};
\node (LT) [below of=TxTT] {$L_{T,T',T''}$};
\node (XXxX) [right of=SSxS] {$(X\times_B X')\times_C X''$};
\node (TTxT) [right of=XXxX] {$(T\times_B T')\times_C T''$};
\draw[->] (XxXX) to node {}(SxSS);
\draw[->] (XxXX) to node [swap]{$q$}(TxTT);
\draw[->] (XXxX) to node [swap]{$p$}(SSxS);
\draw[->] (XXxX) to node {}(TTxT);
\draw[->] (LS) to node {}(SxSS);
\draw[->] (LS) to node {$j^{-1}_{S,S',S''}$}(SSxS);
\draw[->] (LT) to node {$k^{-1}_{T,T',T''}$}(TxTT);
\draw[->] (LT) to node {}(TTxT);
\end{tikzpicture}
\end{center}
We thus need to see that the (vertical) composites $(X\times_B X')\times_C X'' \circ L_{S,S',S''}$ and $L_{T,T',T''} \circ X\times_B (X'\times_C X'')$ are isomorphic as spans of spans. We have to consider the following two pullbacks:

\begin{minipage}[b]{0.5\textwidth}
\centering
\begin{tikzpicture}[->,>=stealth',node distance=2cm, auto]
  \node (SSxS) {$(S\times_B S')\times_C S''$};
  \node (L) [left of=SSxS,above of=SSxS] {$L_{S,S',S''}$};
  \node (XXxX) [right of=SSxS,above of=SSxS] {$(X\times_B X')\times_C X''$};
  \node (XXxX1) [right of=L,above of=L] {$(X\times_B X')\times_C X''$};
  \draw[->] (L) to node [swap]{$j^{-1}_{S,S',S''}$} (SSxS);
  \draw[->] (XXxX) to node {$p$} (SSxS);
  \draw[->] (XXxX1) to node [swap]{$j_{S,S',S''}\circ p$} (L);
  \draw[->] (XXxX1) to node {$1$} (XXxX);
\end{tikzpicture}

\end{minipage}
\begin{minipage}[b]{0.5\textwidth}
\centering
\begin{tikzpicture}[->,>=stealth',node distance=2cm, auto]
  \node (TxTT) {$T\times_B (T'\times_C T'')$};
  \node (L) [right of=TxTT,above of=TxTT] {$L_{T,T',T''}$};
  \node (XxXX) [left of=TxTT,above of=TxTT] {$X\times_B (X'\times_C X'')$};
  \node (XxXX1) [left of=L,above of=L] {$X\times_B (X'\times_C X'')$};
  \draw[->] (L) to node {$k^{-1}_{T,T',T''}$} (TxTT);
  \draw[->] (XxXX) to node [swap]{$q$} (TxTT);
  \draw[->] (XxXX1) to node {$k_{T,T',T''}\circ q$} (L);
  \draw[->] (XxXX1) to node [swap]{$1$} (XxXX);
\end{tikzpicture}

\end{minipage}
Thus we need to show the existence of an isomorphism of spans of spans for the following diagram:
\begin{center}
\begin{tikzpicture}[->,>=stealth',node distance=4cm, auto]
\node (1) {};
\node (SxS1S2) [above of=1] {$S\times_B (S'\times_C S'')$};
\node (XXxX) [below left of=1] {$(X\times_B X')\times_C X''$};
\node (XxXX) [below right of=1] {$X\times_B (X'\times_C X'')$};
\node (2) [below right of=XXxX] {};
\node (TTxT) [below of=2] {$(T\times_B T')\times_C T''$};
\node (A) [left of=XXxX,xshift=.3cm] {$A$};
\node (D) [right of=XxXX,xshift=-.3cm] {$D$};
\draw[->,bend right] (SxS1S2) to node {}(A);
\draw[->] (SxS1S2) to node {}(D);
\draw[->] (TTxT) to node {}(A);
\draw[->,bend right] (TTxT) to node{}(D);
\draw[->] (XXxX) to node {$k^{-1}_{S,S',S''}\circ j_{S,S',S''}\circ p$}(SxS1S2);
\draw[->] (XXxX) to node [swap]{$q$}(TTxT);
\draw[->] (XxXX) to node {$j^{-1}_{T,T',T''}\circ k_{T,T',T''}\circ q$}(TTxT);
\draw[->] (XxXX) to node [swap]{$p$}(SxS1S2);
\draw[->,dashed] (XXxX) to node {$\cong$}(XxXX);

\end{tikzpicture}
\end{center}
Note that the arrows $p$ and $q$ always stand for uniquely commutative arrows from a horizontal composite of spans of spans to the two spans it goes between; the two instances of $p$ and $q$ above do not stand for exactly the same arrows, but differ only in their domain and codomain while having the same meaning. We will construct the desired isomorphism of spans of spans similarly to how was done above in other situations. We start with the base subdiagram determined by just $X,X',X'',S,S',$ and $S''$ (drawn with black arrows in the larger diagram below), and construct the relevant pullbacks and limits:

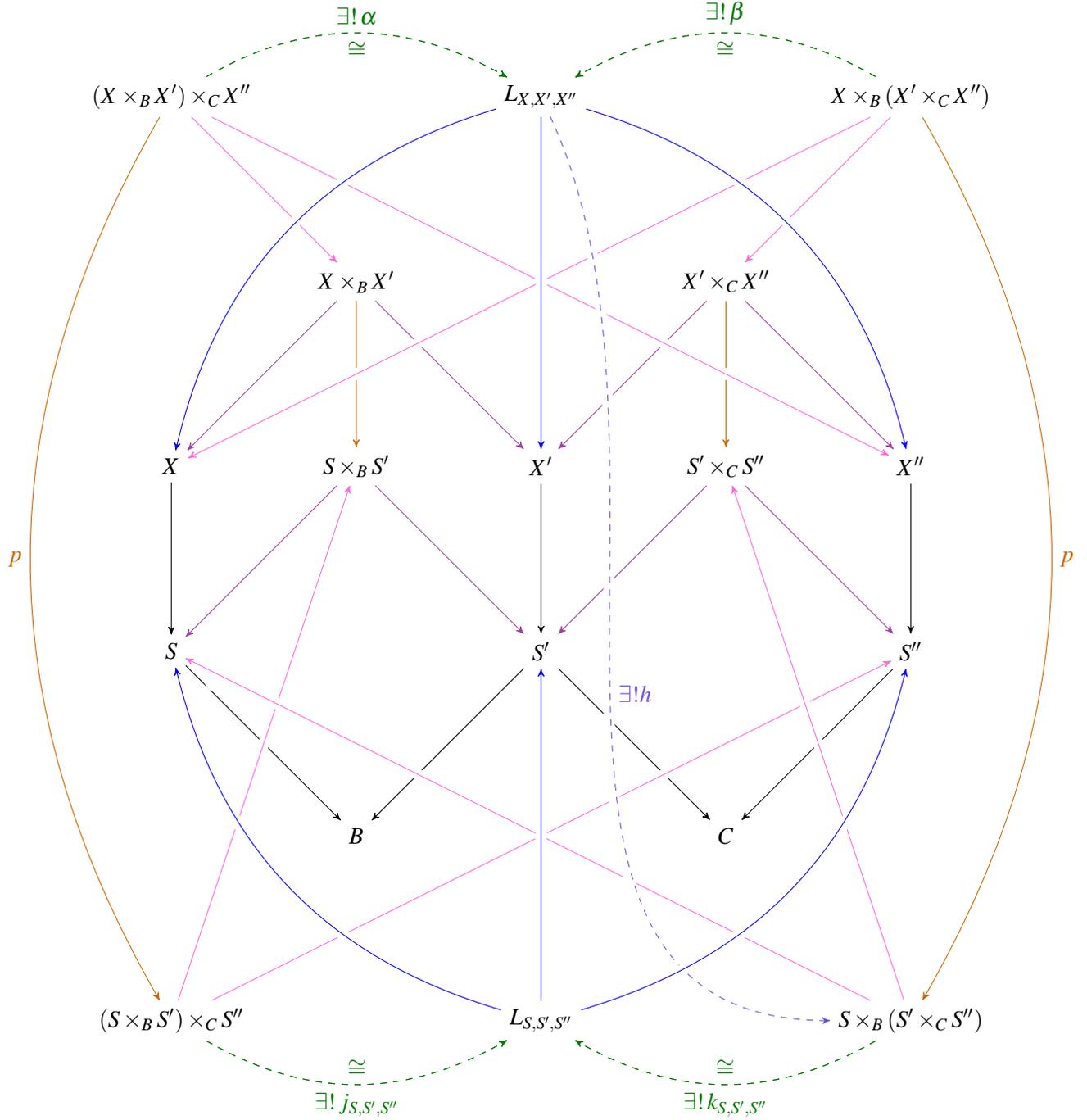
\begin{center}
\begin{tikzpicture}[->,>=stealth',node distance=3cm, auto]
 \node (1) {};
 \node (S) [above of=1] {$S$};
 \node (X) [above of=S] {$X$};
 \node (B) [right of=1] {$B$};
 \node (XX1) [above of=B] {};
 \node (SS1) [above of=XX1] {};
 \node (2) [right of=B] {};
 \node (S1) [above of=2] {$S'$};
 \node (X1) [above of=S1] {$X'$};
 \node (C) [right of=2] {$C$};
 \node (X1X2) [above of=C] {};
 \node (S1S2) [above of=X1X2] {};
 \node (3) [right of=C] {};
 \node (S2) [above of=3] {$S''$};
 \node (X2) [above of=S2] {$X''$};
 \node (SS1) [right of=X] {$S\times_B S'$};
 \node (S1S2) [right of=X1] {$S'\times_C S''$};
\node (XX1) [above of=SS1] {$X\times_B X'$};
\node (X1X2) [above of=S1S2] {$X'\times_C X''$};
\node (LX) [above of=XX1,right of=XX1] {$L_{X,X',X''}$};
\node (LS) [below of=B,right of=B] {$L_{S,S',S''}$};
\node (XXxX) [above of=XX1,left of=XX1] {$(X\times_B X')\times_C X''$};
\node (XxXX) [above of=X1X2,right of=X1X2] {$X\times_B (X'\times_C X'')$};
\node (SSxS) [below of=B,left of=B] {$(S\times_B S')\times_C S''$};
\node (SxSS) [below of=C,right of=C] {$S\times_B (S'\times_C S'')$};
 \draw[->] (X) to node {}(S);
 \draw[->] (S) to node {}(B);
 \draw[->] (S1) to node {}(B);
\draw[->] (X1) to node {}(S1);
\draw[->] (S1) to node {}(C);
\draw[->](S2) to node {}(C);
\draw[->](X2) to node {}(S2);
\draw[->,violet!75!white!] (SS1) to node {}(S);
\draw[->,violet!75!white!] (SS1) to node {}(S1);
\draw[->,violet!75!white!] (S1S2) to node {}(S1);
\draw[->,violet!75!white!] (S1S2) to node {}(S2);
\draw[->,violet!75!white!] (XX1) to node {}(X);
\draw[->,violet!75!white!] (XX1) to node {}(X1);
\draw[->,violet!75!white!] (X1X2) to node {}(X1);
\draw[->,violet!75!white!] (X1X2) to node {}(X2);
\draw[->,orange!80!black] (XX1) to node [swap]{}(SS1);
\draw[->,orange!80!black] (X1X2) to node [swap]{}(S1S2);
\draw[->,pink!55!magenta] (XXxX) to node {}(XX1);
\draw[->,pink!55!magenta] (XxXX) to node {}(X1X2);

\draw[->,bend right,orange!80!black] (XXxX) to node [swap]{$p$}(SSxS);
\draw[->,bend left,orange!80!black] (XxXX) to node {$p$}(SxSS);
\draw[->,dashed,out=30,in=150,green!48!black] (XXxX) to node {$\exists !\,\alpha$} (LX);
\draw[->,dashed,out=30,in=150,green!48!black] (XXxX) to node [swap]{$\cong$} (LX);
\draw[->,dashed,out=150,in=30,green!48!black] (XxXX) to node [swap]{$\exists !\, \beta$} (LX);
\draw[->,dashed,out=150,in=30,green!48!black] (XxXX) to node {$\cong$} (LX);
\draw[->,dashed,out=330,in=210,green!48!black] (SSxS) to node [swap]{$\exists !\,j_{S,S',S''}$} (LS);
\draw[->,dashed,out=330,in=210,green!48!black] (SSxS) to node {$\cong$} (LS);
\draw[->,dashed,out=210,in=330,green!48!black] (SxSS) to node {$\exists !\,k_{S,S',S''}$} (LS);
\draw[->,dashed,out=210,in=330,green!48!black] (SxSS) to node [swap]{$\cong$} (LS);

\draw[->,white,line width=5pt] (XXxX) to node {}(X2);
\draw[->,pink!55!magenta] (XXxX) to node {}(X2);
\draw[->,white,line width=5pt] (XxXX) to node {}(X);
\draw[->,pink!55!magenta] (XxXX) to node {}(X);
\draw[->,white,line width=5pt] (SSxS) to node {}(SS1);
\draw[->,pink!55!magenta] (SSxS) to node {}(SS1);
\draw[->,white,line width=5pt] (SSxS) to node {}(S2);
\draw[->,pink!55!magenta] (SSxS) to node {}(S2);
\draw[->,white,line width=5pt] (SxSS) to node {}(S);
\draw[->,pink!55!magenta] (SxSS) to node {}(S);
\draw[->,white,line width=5pt] (SxSS) to node{}(S1S2);
\draw[->,pink!55!magenta] (SxSS) to node{}(S1S2);
\draw[->,out=-65,in=180,white,line width=5pt] (LX) to node {$\exists !$} (SxSS);
\draw[->,dashed,out=-65,in=180,blue!58!pink!] (LX) to node {$\exists !h$} (SxSS);
\draw[->,bend right,white,line width=5pt] (LX) to node {}(X);
\draw[->,bend right,blue] (LX) to node {}(X);
\draw[->,white,line width=5pt] (LX) to node {}(X1);
\draw[->,blue] (LX) to node {}(X1);
\draw[->,bend left,white,line width=5pt] (LX) to node {}(X2);
\draw[->,bend left,blue] (LX) to node {}(X2);
\draw[->,bend left,white,line width=5pt] (LS) to node {}(S);
\draw[->,bend left,blue] (LS) to node {}(S);
\draw[->,white,line width=5pt] (LS) to node {}(S1);
\draw[->,blue] (LS) to node {}(S1);
\draw[->,bend right,white,line width=5pt] (LS) to node {}(S2);
\draw[->,bend right,blue] (LS) to node {}(S2);
\end{tikzpicture}
\end{center}
Since $(X\times_B X')\times_C X''$ and $X\times_B(X'\times_C X'')$ are seen to be cones over the base diagram, there are unique arrows $\alpha,\beta$ from these iterated pullbacks to the limit of the whole base diagram making their triangles commute. Also, that limit $L_{X,X',X''}$ is seen to be a cone over the cospan determined by $S'\times_C S''$, and hence also over the cospan determined by $S\times_B (S'\times_C S'')$, providing a unique arrow $h$ from $L_{X,X',X''}$ to the latter pullback that makes its relevant triangles commute. By an argument that is formally identical to the ones seen above, the arrows $\alpha,\beta$ are in fact isomorphisms, whose inverses $\alpha^{-1},\beta^{-1}$ also make the diagram commute. In exactly the same fashion, we have isomorphisms $j,k$ of the smaller limit $L_{S,S',S''}$ with the spans $(S\times_B S')\times_C S''$ and $S\times_B (S'\times_C S'')$. We will now prove that $\beta^{-1}\circ \alpha$ is the isomorphism of spans of spans we seek. We show the required equality of paths for the top half of the pertinent diagram; the proof for the bottom half is exactly symmetrical using the equivalent diagram determined by $X,X',X'',T,T',$ and $T''$ instead. We have

$$(X\times_B X')\times_C X'' \xrightarrow{k^{-1}\circ j \circ p} S\times_B (S'\times_C S'') = (X\times_B X')\times_C X'' \xrightarrow{\alpha} L_{X,X',X''} \xrightarrow{h} S\times_B (S'\times_C S'')$$
$$= (X\times_B X')\times_C X'' \xrightarrow{\beta^{-1}\circ \alpha} X\times_B (X'\times_C X'') \xrightarrow{p} S\times_B (S'\times_C S'')$$
as required. We have just completed
\begin{lemma} The associator $a_{A,B,C,D}$ as defined above is a natural isomorphism.
\end{lemma}
\begin{proof} Done.
\end{proof}
\begin{lemma} The associator $a_{A,B,C,D}$ satisfies the pentagon identity.
\end{lemma}
\begin{proof}We need to check, for any object $(S,S',S'',S''')$ of $\mbox{Span}(A,B)\times \,\mbox{Span}(B,C)\times \,\mbox{Span}(C,D) \times \,\mbox{Span}(D,E)$, commutativity of the Mac Lane pentagon, which in this case is

\begin{center}
\begin{tikzpicture}[->,>=stealth',node distance=4cm, auto]
 \node (A) {$S\times_B (S'\times_C (S''\times_D S'''))$};
 \node (B) [right of=A,xshift=2.5cm] {$S\times_B ((S'\times_C S'')\times_D S''')$};
 \node (C) [below of=A] {$(S\times_B S')\times_C (S''\times_D S''')$};
 \node (D) [below of=B] {$(S\times_B (S'\times_C S''))\times_D S'''$};
 \node (E) [below of=C,xshift=3.25cm] {$((S\times_B S')\times_C S'')\times_D S'''$};
 \draw[double equal sign distance, -implies] (A) to node {$\mbox{Id}_S \, ;\, a_{S',S'',S'''}$}(B);
 \draw[double equal sign distance,-implies] (A) to node [swap]{$a_{S,S',S''\times_D S'''}$}(C);
 \draw[double equal sign distance,-implies] (C) to node [swap]{$a_{S\times_B S',S'',S'''}$}(E);
 \draw[double equal sign distance,-implies] (B) to node {$a_{S,S'\times_C S'',S'''}$}(D);
 \draw[double equal sign distance,-implies] (D) to node {$a_{S,S',S''}\, ; \,\mbox{Id}_{S'''}$}(E);

\end{tikzpicture}
\end{center}
In terms of vertical and horizontal compositions, the commutativity of this diagram implies the existence of an isomorphism of spans of spans for the following diagram, from the vertical composite on the left to that on the right:
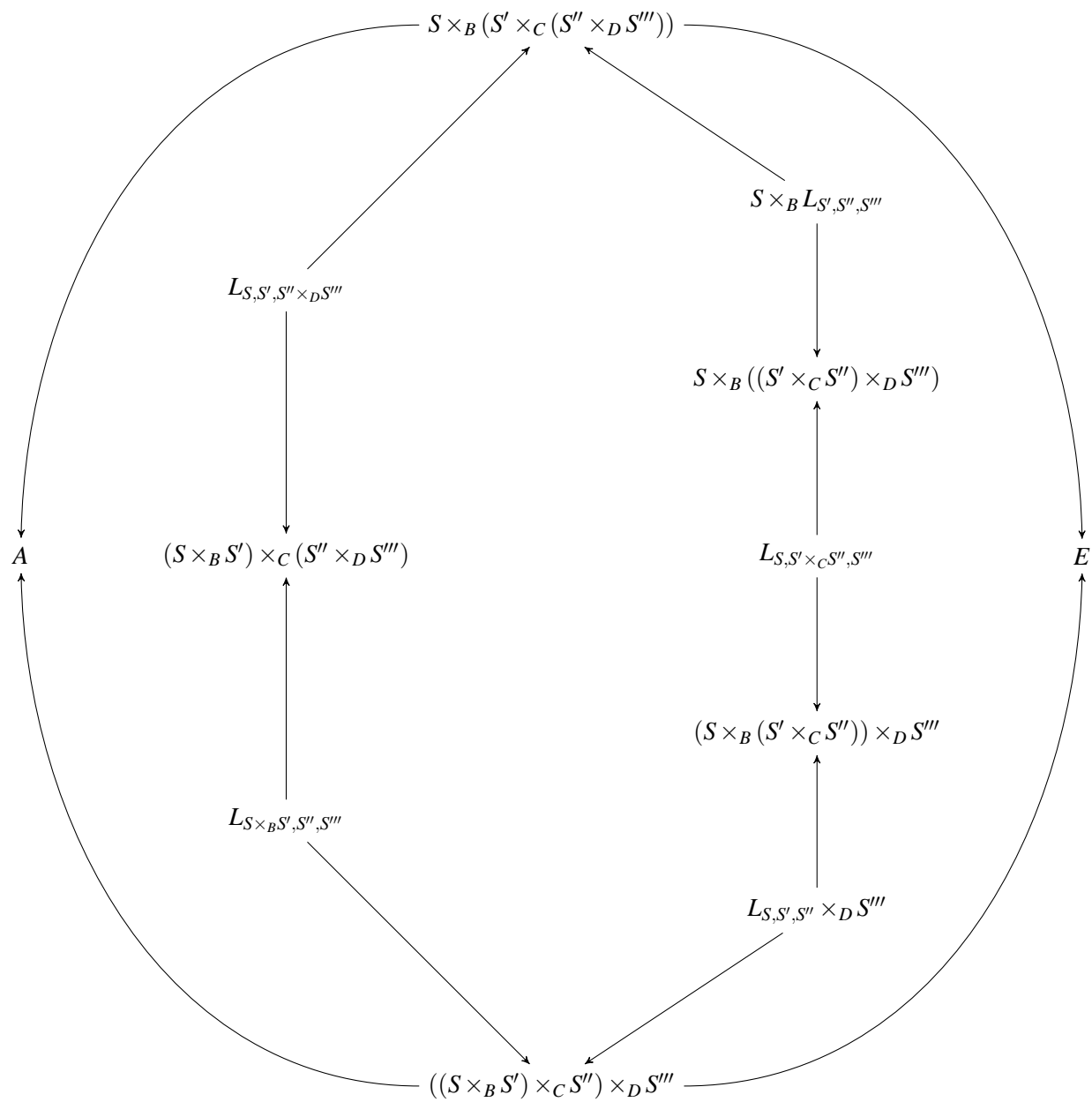
\begin{center}
\begin{tikzpicture}[->,>=stealth',node distance=4cm, auto]
\node (SxS123) {$S\times_B (S'\times_C (S''\times_D S'''))$};
\node (L1) [below of=SxS123, left of=SxS123] {$L_{S,S',S''\times_D S'''}$};
\node (SS) [below of=L1] {$(S\times_B S')\times_C (S''\times_D S''')$};
\node (L2) [below of=SS] {$L_{S\times_B S',S'',S'''}$};
\node (S12xS3) [below of=L2, right of=L2] {$((S\times_B S')\times_C S'')\times_D S'''$};
\node (SL3) [below of=SxS123,right of=SxS123,yshift=1.33cm] {$S\times_B L_{S',S'',S'''}$};
\node (SS1) [below of=SL3,yshift=1.33cm] {$S\times_B ((S'\times_C S'')\times_D S''')$};
\node (L4) [below of=SS1,yshift=1.33cm] {$L_{S,S'\times_C S'',S'''}$};
\node (SS2) [below of=L4,yshift=1.33cm] {$(S\times_B (S'\times_C S''))\times_D S'''$};
\node (L5) [below of=SS2,yshift=1.33cm] {$L_{S,S',S''}\times_D S'''$};
\node (A) [left of=SS] {$A$};
\node (E) [right of=L4] {$E$};
\draw[->,out=180,in=89] (SxS123) to node {}(A);
\draw[->,out=0,in=91] (SxS123) to node {}(E);
\draw[->,out=180,in=271] (S12xS3) to node {}(A);
\draw[->,out=0,in=269] (S12xS3) to node {}(E);
\draw[->] (L1) to node {}(SxS123);
\draw[->] (L1) to node {}(SS);
\draw[->] (L2) to node {}(SS);
\draw[->] (L2) to node {}(S12xS3);
\draw[->] (SL3) to node {}(SxS123);
\draw[->] (SL3) to node {}(SS1);
\draw[->] (L4) to node {}(SS1);
\draw[->] (L4) to node {}(SS2);
\draw[->] (L5) to node {}(SS2);
\draw[->] (L5) to node {}(S12xS3);
\end{tikzpicture}
\end{center}
We will show that both (vertical composite) paths from $S \times_B (S'\times_C (S'' \times_D S'''))$ to $((S \times_B S')\times_C S'')\times_D S'''$ are equal to the invertible span of spans $L_{S,S',S'',S'''}$, which we define to be the limit of the diagram

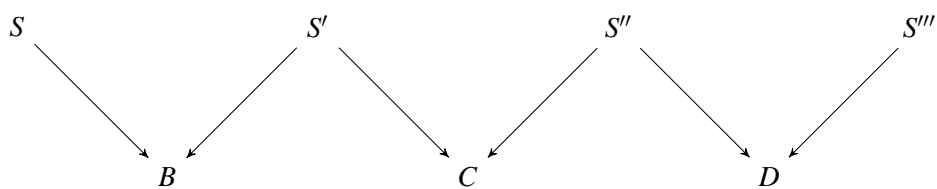
\begin{center}
\begin{tikzpicture}[->,>=stealth',node distance=2cm, auto]
 \node (S) {$S$};
 \node (B) [below of=S,right of=S] {$B$};
 \node (S1) [above of=B,right of=B] {$S'$};
 \node (C) [below of=S1,right of=S1] {$C$};
 \node (S2) [above of=C,right of=C] {$S''$};
 \node (D) [below of=S2,right of=S2] {$D$};
 \node (S3) [above of=D,right of=D] {$S'''$};
 \draw[->] (S) to node {} (B);
 \draw[->] (S1) to node {}(B);
 \draw[->] (S1) to node {}(C);
\draw[->] (S2) to node {}(C);
 \draw[->] (S2) to node {}(D);
\draw[->] (S3) to node {}(D);

\end{tikzpicture}
\end{center}

As the diagrams involved in this argument are difficult to draw in a clean way, we will omit them. First consider the two sided part of the pentagon, i.e. the pullback $L_{S,S',S''\times_D S'''}\times_C L_{S\times S',S'',S'''}$ (this pullback is originally over $(S\times_B S')\times_C (S''\times_D S''')$, but as the latter pullback is done over $C$, the original pullback may also be reduced to be done over $C$, as seen earlier in this paper). One easily sees that $L_{S,S',S''\times_D S'''}\times_C L_{S\times_B S',S'',S'''}$ is a cone over the base diagram determined by $S,S',S'',S'''$, so there is a uniquely commuting arrow $\alpha$ from $L_{S,S',S''\times_D S'''}\times_C L_{S\times_B S',S'',S'''}$ to $L_{S,S',S'',S'''}$. Also, there are uniquely commuting arrows $f$ and $g$ from $L_{S,S',S'',S'''}$ to $S\times_B S'$ and $S''\times_D S'''$, respectively, by the universal property of pullbacks. The existence of $f$ and $g$ then provides further unique arrows $h,i$ from $L_{S,S',S'',S'''}$ to $L_{S,S',S''\times_D S'''}$ and $L_{S\times_B S',S'',S'''}$, respectively, by the universal property of limits. Finally, $h$ and $i$ together provide a unique arrow from $L_{S,S',S'',S'''}$ to $L_{S,S',S''\times_D S'''}\times_C L_{S\times_B S',S'',S'''}$; one easily sees by the usual argument that this last arrow is an inverse for $\alpha$, so that $\alpha$ is a canonical isomorphism.

For the three sided part of the pentagon, we have to deal with an iterated pullback, which may be parenthesized two different ways; we have shown earlier that both parenthesizations are equal as spans of spans (associativity of vertical composition), so we will arbitrarily choose one of them to work with. Consider $$((S\times_B L_{S',S'',S'''})\times_B L_{S,S'\times_B S'',S'''})\times_D (L_{S,S',S''}\times_D S''').$$
For brevity, call this pullback $P$. It is clear that this is a cone on the base diagram determined by $S,S',S'',S'''$ if one follows the two arrows to the `factors' of that pullback, which are themselves pullbacks involving limits, down to the bottom. This gives a unique arrow $\beta$ from $P$ to $L_{S,S',S'',S'''}$. Next note that there are uniquely commuting arrows from $L_{S,S',S'',S'''}$ to $L_{S,S'\times_C S'',S'''}$, to $S\times_B L_{S',S'',S'''}$, and to $L_{S,S',S''}\times_D S'''$, by the universal properties of limits and pullbacks. Next, the unique arrows from $L_{S,S',S'',S'''}$ to $S\times_B L_{S',S'',S'''}$ and $L_{S,S'\times_C S'',S'''}$ give a uniquely commuting arrow to the pullback $(S\times_B L_{S',S'',S'''})\times_B L_{S,S'\times_C S''}$; that arrow along with the unique arrow to $L_{S,S',S''}\times_D S'''$ provide a uniquely commuting arrow from $L_{S,S',S'',S'''}$ to $P$. This last arrow is seen to be an inverse of $\beta$, again by the usual argument used throughout this paper. Therefore, $P$ and $L_{S,S',S'',S'''}$ are canonically isomorphic. This gives us the pentagon identity, because the property of being canonically isomorphic is transitive, so that the two sided part equals the three sided part of the pentagon.
\end{proof}
\section{Unitors}
The last thing required to see $\mbox{Span}_2(\mathcal{C})$ is a bicategory is to define a pair of right and left `unitor' natural isomorphisms,
$$r_{A,B}:\, ;_{A,A,B} \circ (I_A \times \mbox{Id}) \Longrightarrow \mbox{Id},$$
$$l_{A,B}:\, ;_{A,B,B} \circ (\mbox{Id}\times I_B) \Longrightarrow \mbox{Id},$$
satisfying a coherence law called the triangle identity. Here the functor $;_{A,A,B}\circ (I_A \times \mbox{Id})$ is understood to be going from $\mbox{Span}(A,B)$ to itself, sending an object $_{A} S_B$ to the horizontal composite $_{A} A_A \, ; \, _{A} S_B$, and sending a morphism $_{S}X_T$ to $\mbox{Id}_{_{A}A_A}\,;\,_{S}X_T$; similarly for $;_{A,B,B} \circ (\mbox{Id}\times I_B)$. The next diagram shows the canonical isomorphism $_{A} A_A \times_A \,_{A}S_B \cong \, _{A}S_B$:

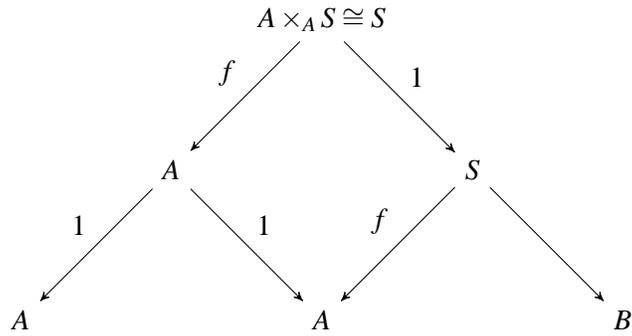
\begin{center}
\begin{tikzpicture}[->,>=stealth',node distance=2cm, auto]
 \node (A) {$A$};
 \node (A1) [above of=A,right of=A] {$A$};
 \node (A2) [below of=A1,right of=A1] {$A$}; 
 \node (S) [above of=A2,right of=A2] {$S$};
 \node (B) [below of=S,right of=S] {$B$};
 \node (AS) [above of=A1,right of=A1] {$A\times_A S\cong S$};
 \draw[->] (A1) to node [swap]{$1$} (A);
 \draw[->] (A1) to node {$1$} (A2);
 \draw[->] (S) to node [swap]{$f$}(A2);
\draw[->] (S) to node {}(B);
 \draw[->] (AS) to node [swap]{$f$}(A1);
\draw[->] (AS) to node {$1$}(S);
\end{tikzpicture}
\end{center}
Using this, we also obtain a canonical isomorphism $\mbox{Id}_{_{A}A_A}\,;\, _{S}X_T\cong \,_{S}X_T$, as is easily checked. With this in mind, we define the component $r_{S}$ to be the span of spans:

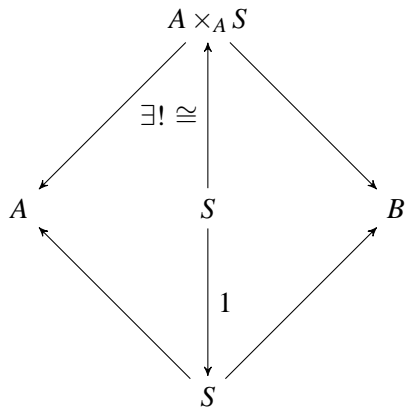
\begin{center}
\begin{tikzpicture}[->,>=stealth',node distance=2.5cm, auto]
 \node (A) {$A$};
 \node (AS) [above of=A,right of=A] {$A\times_A S$};
 \node (B) [below of=AS,right of=AS] {$B$}; 
 \node (S) [right of=A] {$S$};
 \node (S1) [below of=A,right of=A] {$S$};
 \draw[->] (AS) to node {} (A);
 \draw[->] (AS) to node {} (B);
 \draw[->] (S) to node {$\exists !\,\cong$}(AS);
\draw[->] (S) to node {$1$}(S1);
 \draw[->] (S1) to node {}(A);
\draw[->] (S1) to node {}(B);
\end{tikzpicture}
\end{center}

Similarly, $l_S$ is defined to be

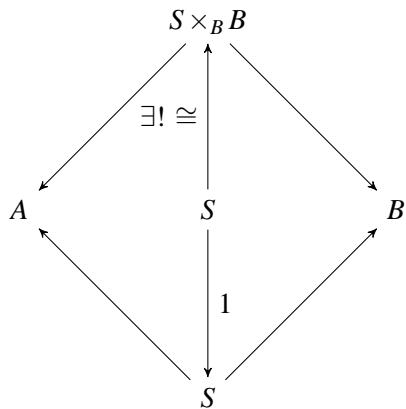
\begin{center}
\begin{tikzpicture}[->,>=stealth',node distance=2.5cm, auto]
 \node (A) {$A$};
 \node (SB) [above of=A,right of=A] {$S\times_B B$};
 \node (B) [below of=SB,right of=SB] {$B$}; 
 \node (S) [right of=A] {$S$};
 \node (S1) [below of=A,right of=A] {$S$};
 \draw[->] (SB) to node {} (A);
 \draw[->] (SB) to node {} (B);
 \draw[->] (S) to node {$\exists !\,\cong$}(SB);
\draw[->] (S) to node {$1$}(S1);
 \draw[->] (S1) to node {}(A);
\draw[->] (S1) to node {}(B);
\end{tikzpicture}
\end{center}

These spans of spans are easily seen to be invertible. We have to check the naturality square:

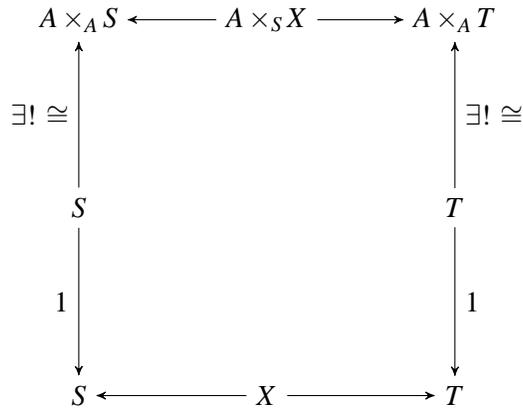
\begin{center}
\begin{tikzpicture}[->,>=stealth',node distance=2.5cm, auto]
\node (AS) {$A\times_A S$};
\node (AX) [right of=AS] {$A\times_S X$};
\node (AT) [right of=AX] {$A\times_A T$};
\node (S) [below of=AS] {$S$};
\node (S1) [below of=S] {$S$};
\node (X) [right of=S1] {$X$};
\node (T1) [right of=X] {$T$};
\node (T) [above of=T1] {$T$};
\draw[->] (AX) to node {}(AS);
\draw[->] (AX) to node {}(AT);
\draw[->] (S) to node {$\exists !\, \cong$}(AS);
\draw[->] (S) to node [swap]{$1$} (S1);
\draw[->] (T) to node [swap]{$\exists !\, \cong$} (AT);
\draw[->] (T) to node {$1$} (T1);
\draw[->] (X) to node {} (S1);
\draw[->] (X) to node {}(T1);

\end{tikzpicture}
\end{center}
The vertical composite obtained from going along the left and then bottom of this square is the pullback over $S$ of the interior of:

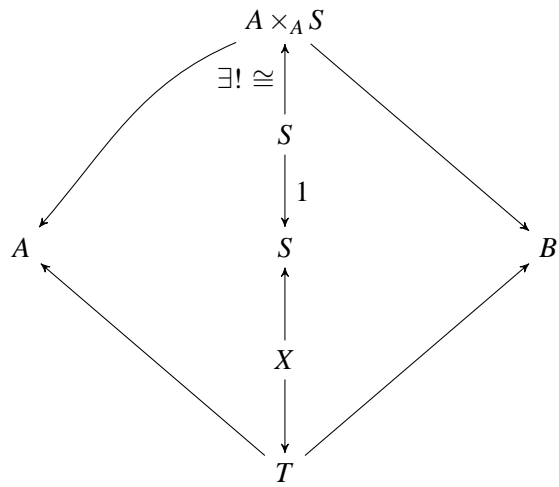
\begin{center}
\begin{tikzpicture}[->,>=stealth',node distance=1.5cm, auto]
 \node (A) {$A$};
 \node (S) [right of=A,xshift=2cm] {$S$};
 \node (B) [right of=S,xshift=2cm] {$B$};
 \node (S1) [above of=S] {$S$};
 \node (AS) [above of=S1] {$A\times_A S$};
 \node (X) [below of=S] {$X$};
 \node (T) [below of=X] {$T$};
\draw[->,out=203,in=49] (AS) to node {}(A);
\draw[->] (AS) to node {}(B);
\draw[->] (S1) to node {$\exists !\,\cong$}(AS);
\draw[->] (S1) to node {$1$}(S);
\draw[->] (X) to node {}(S);
\draw[->] (X) to node {}(T);
\draw[->] (T) to node {}(A);
\draw[->] (T) to node {}(B);
\end{tikzpicture}
\end{center}
Taking the pullback (noting that one leg is an identity arrow), this becomes 

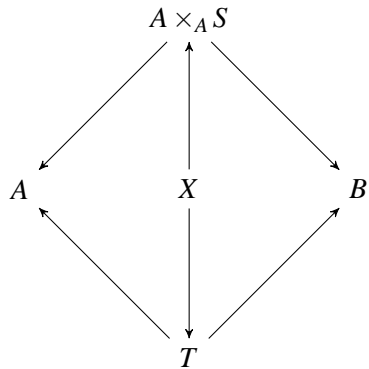
\begin{center}
\begin{tikzpicture}[->,>=stealth',node distance=2.25cm, auto]
\node (A) {$A$};
\node (X) [right of=A] {$X$};
\node (B) [right of=X] {$B$};
\node (AS) [above of=X] {$A\times_A S$};
\node (T) [below of=X] {$T$};
\draw[->] (AS) to node {}(A);
\draw[->] (AS) to node {}(B);
\draw[->] (X) to node {}(AS);
\draw[->] (X) to node {}(T);
\draw[->] (T) to node {}(A);
\draw[->] (T) to node {}(B);

\end{tikzpicture}
\end{center}

On the other hand, going along the top of the square and then the right side leads to the diagram
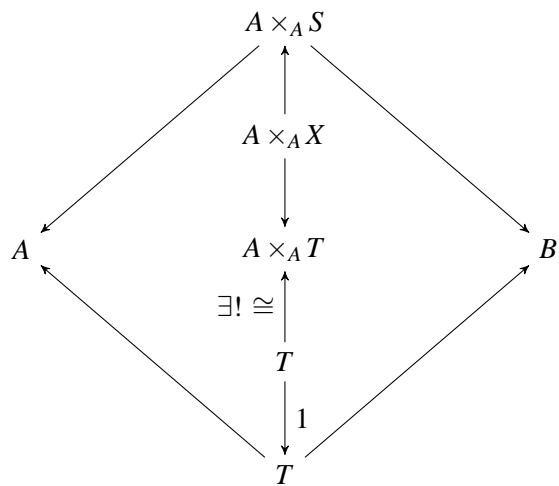
\begin{center}
\begin{tikzpicture}[->,>=stealth',node distance=1.5cm, auto]
 \node (A) {$A$};
 \node (AT) [right of=A,xshift=2cm] {$A\times_A T$};
 \node (B) [right of=AT,xshift=2cm] {$B$};
 \node (AX) [above of=AT] {$A\times_A X$};
 \node (AS) [above of=AX] {$A\times_A S$};
 \node (T) [below of=AT] {$T$};
 \node (T1) [below of=T] {$T$};
\draw[->] (AS) to node {}(A);
\draw[->] (AS) to node {}(B);
\draw[->] (AX) to node {}(AS);
\draw[->] (AX) to node {}(AT);
\draw[->] (T) to node {$\exists !\,\cong$}(AT);
\draw[->] (T) to node {$1$}(T1);
\draw[->] (T1) to node {}(A);
\draw[->] (T1) to node {}(B);
\end{tikzpicture}
\end{center}
Taking this pullback (using the inverse of the given canonical isomorphism from $T$ to $A\times_A T$) this becomes
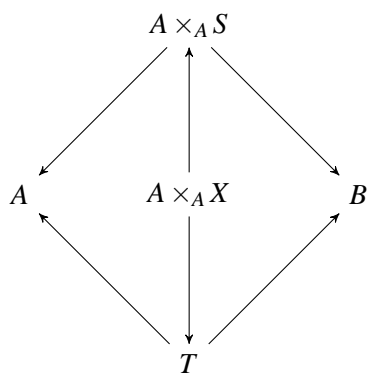
\begin{center}
\begin{tikzpicture}[->,>=stealth',node distance=2.25cm, auto]
\node (A) {$A$};
\node (AX) [right of=A] {$A\times_A X$};
\node (B) [right of=X] {$B$};
\node (AS) [above of=X] {$A\times_A S$};
\node (T) [below of=X] {$T$};
\draw[->] (AS) to node {}(A);
\draw[->] (AS) to node {}(B);
\draw[->] (AX) to node {}(AS);
\draw[->] (AX) to node {}(T);
\draw[->] (T) to node {}(A);
\draw[->] (T) to node {}(B);

\end{tikzpicture}
\end{center}
Since $A\times_A X \cong X$ is a canonical isomorphism, this last diagram is equivalent to the one we arrived at going the other way around the naturality square. This proves that the right unitor $r_{A,B}$ really is a natural isomorphism; the proof of this fact for $l_{A,B}$ is analogous. We have just completed
\begin{lemma} The right and left unitors as defined above are natural isomorphisms.
\end{lemma}
\begin{proof} Above.
\end{proof}

For the triangle identity, we need to see that for any pair of spans $_{A}S_B,_{B}S'_{C}$, the following diagram commutes:

\begin{center}
\begin{tikzpicture}[->,>=stealth',node distance=2.5cm, auto]
\node (SBS1) {$S\times_B (B\times_B S')$};
\node (SBS12) [right of=SBS1,xshift=1cm] {$(S\times_B B)\times_B S'$};
\node (SS1) [below of=SBS1,xshift=1.75cm] {$S\times_B S'$};
\draw[->] (SBS1) to node {$a_{S,B,S'}$}(SBS12);
\draw[->] (SBS1) to node [swap]{$S\times r_{B,S'}$}(SS1);
\draw[->] (SBS12) to node {$l_{S,B}\times S'$}(SS1);
\end{tikzpicture}
\end{center}

The left side of this triangle is the invertible span of spans
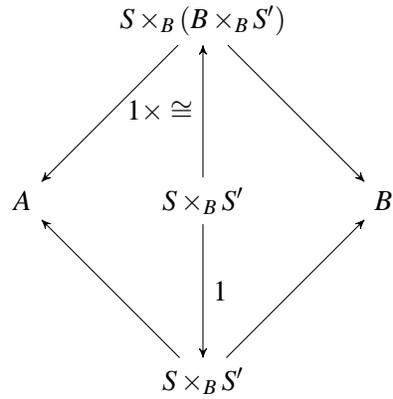
\begin{center}
\begin{tikzpicture}[->,>=stealth',node distance=2.4cm, auto]
\node (A) {$A$};
\node (SS1) [right of=A] {$S\times_B S'$};
\node (B) [right of=SS1] {$B$};
\node (SBS1) [above of=SS1] {$S\times_B (B\times_B S')$};
\node (SS12) [below of=SS1] {$S\times_B S'$};
\draw[->] (SBS1) to node {}(A);
\draw[->] (SBS1) to node {}(B);
\draw[->] (SS12) to node {}(A);
\draw[->] (SS12) to node {}(B);
\draw[->] (SS1) to node {$1\times \cong$}(SBS1);
\draw[->] (SS1) to node {$1$} (SS12);
\end{tikzpicture}
\end{center}

Going along the top of the triangle and then the right side, we consider the diagram
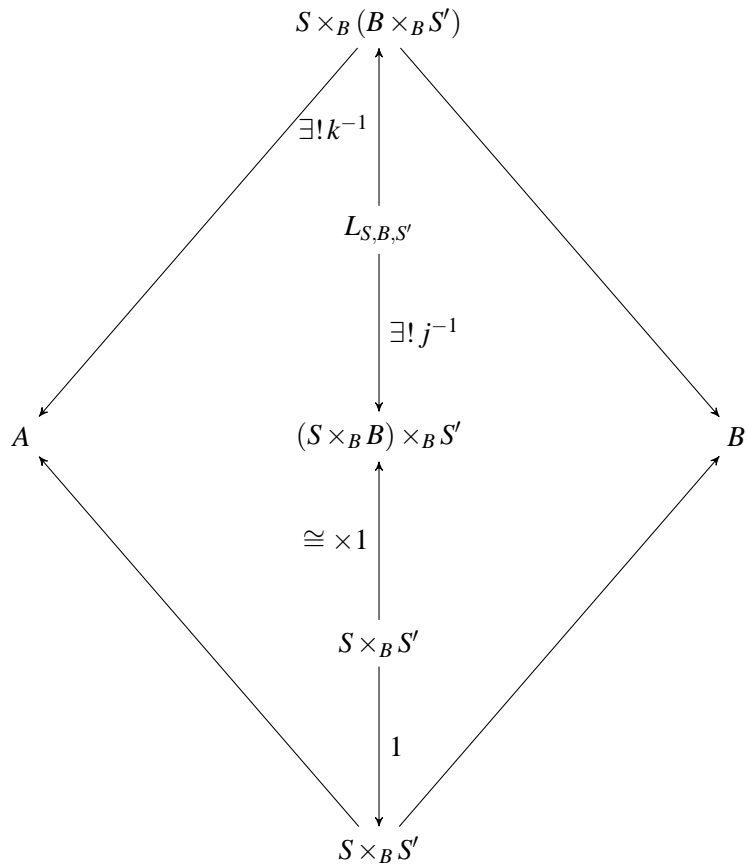
\begin{center}
\begin{tikzpicture}[->,>=stealth',node distance=2.75cm, auto]
 \node (A) {$A$};
 \node (SBS) [right of=A,xshift=2cm] {$(S\times_B B)\times_B S'$};
 \node (B) [right of=SBS,xshift=2cm] {$B$};
 \node (L) [above of=SBS] {$L_{S,B,S'}$};
 \node (SBS1) [above of=L] {$S\times_B (B\times_B S')$};
 \node (SS1) [below of=SBS] {$S\times_B S'$};
 \node (SS2) [below of=SS1] {$S\times_B S'$};
\draw[->] (SBS1) to node {}(A);
\draw[->] (SBS1) to node {}(B);
\draw[->] (L) to node {$\exists !\,k^{-1}$}(SBS1);
\draw[->] (L) to node {$\exists !\,j^{-1}$}(SBS);
\draw[->] (SS1) to node {$\cong \times 1$}(SBS);
\draw[->] (SS1) to node {$1$}(SS2);
\draw[->] (SS2) to node {}(A);
\draw[->] (SS2) to node {}(B);
\end{tikzpicture}
\end{center}
To deal with this diagram, recall the basic fact about pullbacks that if both legs of the pullback diagram are isomorphisms, the pullback object is canonically isomorphic to either object below:
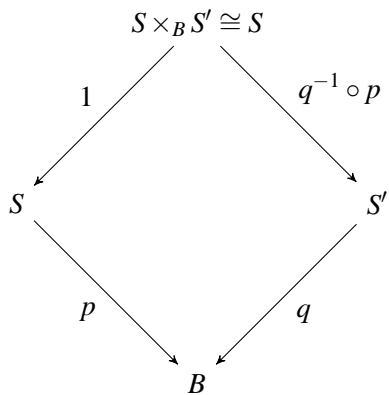
\begin{center}
\begin{tikzpicture}[->,>=stealth',node distance=2.4cm, auto]
\node (S) {$S$};
\node (SS1) [above of=S,right of=S] {$S\times_B S'\cong S$};
\node (S1) [below of=SS1,right of=SS1] {$S'$};
\node (B) [below of=S,right of=S] {$B$};
\draw[->] (S) to node [swap]{$p$}(B);
\draw[->] (S1) to node {$q$}(B);
\draw[->] (SS1) to node [swap]{$1$}(S);
\draw[->] (SS1) to node {$q^{-1}\circ p$}(S1);
\end{tikzpicture}
\end{center}
(We omit the analogous diagram showing $S\times_B S'\cong S'$.) Therefore, the pullback of $L_{S,B,S'}$ and $S\times_B S'$ over $(S\times_B B)\times_B S'$ is canonically isomorphic to $S\times_B S'$, and we again obtain the diagram seen earlier from taking the left side of the triangle, as required. This establishes
\begin{lemma} The right and left unitors satisfy the triangle identity.
\end{lemma}
\begin{proof} See discussion above.
\end{proof}
\section{Main result}
\begin{theorem} For any category $\mathcal{C}$ with pullbacks and a terminal object, $\mbox{Span}_2(\mathcal{C})$ is a bicategory.
\end{theorem}

\begin{proof} Taking objects of $\mbox{Span}_2(\mathcal{C})$ to be the objects of $\mathcal{C}$, morphisms to be spans in $\mathcal{C}$ (objects of the hom-categories $\mbox{Span}(A,B)$ as $A,B$ range over the objects of $\mathcal{C}$), and 2-cells to be isomorphism classes of spans of spans in $\mathcal{C}$ (morphisms of the hom-categories $\mbox{Span}(A,B)$), the contents of Lemmas 1 through 7 establish that $\mbox{Span}_2(\mathcal{C})$ is a bicategory.\\\\
\end{proof}
\emph{Remark:} This result is more flexible than as stated, since given a category $\mathcal{C}$ with a terminal object in which not all pullbacks exist, we can restrict our attention to the subcollection of cospans for which pullbacks \emph{do} exist. This generates a subcategory in which spans are automatically composable, and then Lemmas 1 through 7 all hold to give a bicategory.

\section{An Application: Cobordisms}
The main result of this paper can be used to reduce the work needed in proving there is a bicategory where objects are $(n-2)$-dimensional manifolds without boundary, morphisms are $(n-1)$-manifolds with boundary acting as cobordisms, and 2-morphisms are $n$-manifolds with corners acting as cobordisms of cobordisms.  The idea is that a cobordism of manifolds $M \rightarrow X \leftarrow N$ is a particular kind of cospan, where the arrows are inclusion maps; furthermore, a cobordism of cobordisms is a particular cospan of cospans, again with inclusion maps as arrows. By working with the opposite category, we can apply our result about spans and spans of spans to deduce we have a bicategory, however, there is the obstacle that pushouts do not generally exist in the category of manifolds with corners. This is dealt with simply by restricting attention to certain manifolds with corners for which these pushouts do exist, namely, appropriately `collared' manifolds. We now give a brief review of the definitions and results involved, following Laures \cite{Lau}. Note that a definition of cobordism bicategory also following Laures' framework is given in Schommer-Pries' \cite{S-P}, albeit without details verifying all the axioms of a bicategory. Cobordisms with corners are also addressed in Morton's \cite{Mor} in the context of double bicategories.\\
\\
A \emph{differentiable manifold with corners} is a second-countable Hausdorff space $X$ with a maximal atlas of compatible charts
$$\phi : U \rightarrow [0,\infty)^n,$$
where $U$ is open in $X$. Compatibility means that for any two charts $(\phi_1,U_1),(\phi_2,U_2)$, the composite 
$$\phi_2 \circ \phi_1^{-1} : \phi_1(U_1 \cap U_2) \rightarrow \phi_2(U_1 \cap U_2)$$
is a diffeomorphism. For any $U\subset X$ and $x\in U$, it can be shown that the number of zeros in the coordinate representation $\phi(x)$ is the same for any chart $(\phi,U)$. This number, denoted $c(x)$ and called the \emph{depth of} $x$, measures the degree to which $x$ is a corner point; depth 0 points are in the interior of $X$, depth 1 points are on the boundary of $X$, depth 2 points resemble the corner point of a quadrant in $\mathbb{R}^2$, and so on. A \emph{connected face} is the closure of some component of the boundary, $\{x\in X \,|\,c(x)=1\}$, and a \emph{face} is a disjoint union of connected faces. A \emph{manifold with faces} is a manifold with corners such that each $x$ belongs to $c(x)$ many connected faces. As J\"anich points out in \cite{Jan}, faces of a manifold with faces are in fact manifolds with faces themselves.\\
\\
Now an $\langle n \rangle $-manifold is a manifold with faces together with an ordered $n$-tuple $(\partial_0 X,\partial_1 X,\ldots,\partial_{n-1} X)$ of faces of $X$ satisfying 
$$(1)\,\,\,\, \partial_0 \cup \cdots \cup \partial_{n-1} X = \partial X,\,\,\,\mbox{and}$$
$$(2)\,\,\,\, \partial_i X \cap \partial_j X \,\,\mbox{is a face of}\,\partial_i X\,\mbox{and of}\,\partial_j X\,\mbox{for all}\,i\neq j.$$

Laures showed that given an $\langle n \rangle$-manifold $X$, we can obtain in a canonical way a functor $X: \underline{2}^n \rightarrow \mbox{Top}$, where $\underline{2}$ denotes the category with two objects and one arrow,
$$0\rightarrow 1.$$
This functor is explained below shortly. For an object $a=(a_0,\ldots,a_{n-1})\in \underline{2}^n$, denote its complementary object $(1,\ldots,1)-a$ by $a'$. Also let $e_i$ denote the standard $i^{\tiny{\mbox{th}}}$ basis vector in $\mathbb{R}^n$, so $e_1=(1,0,\ldots,0)$, etc. Then the action of the functor $X$ on objects $a$ is defined by
$$X(a)=\bigcap_{i\in\{i\,|\,a\leq e'_i\}} \partial_i X\,\,\,\,\mbox{for}\,\,\,\, a\neq 0',$$
$$X(0')=X.$$
A morphism $b<a$ in the poset category $\underline{2}^n$ is sent via the functor $X$ to the natural inclusion map between topological spaces, $X(b)$ into $X(a)$. Now, the target category of $X$ can actually be made more specific, from general topological spaces to manifolds with corners, without altering how the functor is defined. This is because the action of $X$ on objects is given by an intersection of faces of $X$, and by condition $(2)$ of being an $\langle n \rangle$-manifold, this intersection is itself a face of each face being intersected. Since faces are manifolds with faces, the resulting intersection is, in particular, a manifold with corners. Furthermore, using the definition of smooth map between manifolds with corners found in \cite{Joy}, the inclusion map $i:\partial X \rightarrow X$ is smooth; this implies that for an $\langle n \rangle $-manifold $X$, each inclusion $i_j : \partial_j X \rightarrow X$ is smooth, and that $i_{j,k}: \partial_j X \rightarrow \partial_k X$ is smooth, when this inclusion makes sense.\\\\
Laures went on to show that there is a collared version of this functor, denoted $C$, defined on objects $a\in \underline{2}^n$ by $C(a)=\mathbb{R}^n_+(a')\times X(a)$, and on morphisms $a < b$ by $C(a<b):\mathbb{R}^n_+(a') \times X(a) \hookrightarrow \mathbb{R}^n_+(b') \times X(b)$, where this arrow has the property of being a topological embedding such that its restriction to $\mathbb{R}^n_+(b') \times X(a)$ is the inclusion map $\mbox{id} \times X(a<b)$. This functor $C$, like $X$, can also be regarded as having manifolds with corners as its target, as opposed to general topological spaces.\\\\
Now suppose $X$ is an $\langle n \rangle $-manifold such that each face of the $n$-tuple $(\partial_0 X,\partial_1 X,\ldots,\partial_{n-1} X)$ is itself a disjoint union, $\partial_i X=s_iX \sqcup t_iX$ (we think of $s_iX$ as the `source' of $\partial_i X$, and $t_iX$ as the `target'). We call such an $X$ a \emph{cubical} $\langle n \rangle $-\emph{manifold}, because the collection of all possible intersections of face components, together with the interior of entire manifold $X$, make up a total of $3^n$ objects that can be naturally arranged into an $n$-dimensional hypercube diagram (proved below). As an illustration of this, consider the solid unit square, the most basic cubical $\langle 2 \rangle $-manifold. Here $s_1X, t_1X$ are opposite closed segments of unit length, as are $s_2X,t_2X$. Taking all possible intersections of these face components yields 4 corners (from pairwise intersection of different non-disjoint components) and 4 edges (the face components themselves); including the interior of $X$ totals $4+4+1=9=3^2$ objects. Following this example, it is clear that any cubical $\langle n \rangle$-manifold can have its $2n$ face components regarded as the $2n$ faces of an $n$-dimensional hypercube, with each pair of corresponding face components in the decomposition being opposite each other. Then using the well known formula for the total number of cellular components to an $n$-hypercube,
$$\sum_{i=0}^n 2^{n-i} {n \choose i} = 3^n,$$
we have a proof of the above claim about cubical $\langle n \rangle$-manifolds, since all possible intersections of face components result in all possible cellular components except the interior.
The case of a cubical $\langle 2 \rangle $-manifold with its parts arranged into a square is illustrated by the diagram below.
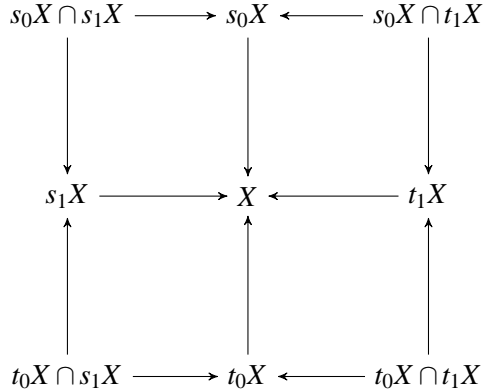
\begin{center}
\begin{tikzpicture}[->,>=stealth',node distance=2.4cm, auto]
\node (X0) {$s_0X$};
\node (1) [left of=X0] {$s_0X \cap s_1X$};
\node (3) [right of=X0] {$s_0X \cap t_1X$};
\node (X) [below of=X0] {$X$};
\node (4) [left of=X] {$s_1X$};
\node (6) [right of=X] {$t_1X $};
\node (X1) [below of=X] {$t_0X $};
\node (7) [left of=X1] {$t_0X \cap s_1X $};
\node (9) [right of=X1] {$t_0X \cap t_1X$};
\draw[->] (1) to node {}(X0);
\draw[->] (3) to node {}(X0);
\draw[->] (1) to node {}(4);
\draw[->] (X0) to node {}(X);
\draw[->] (3) to node {}(6);
\draw[->] (4) to node {}(X);
\draw[->] (6) to node {}(X);
\draw[->] (7) to node {}(4);
\draw[->] (X1) to node {}(X);
\draw[->] (9) to node {} (6);
\draw[->] (7) to node {}(X1);
\draw[->] (9) to node {} (X1);
\end{tikzpicture}
\end{center} 

We can summarize the above paragraph by saying cubical $\langle n \rangle$-manifolds $X$ give a canonical functor $X: K^n \rightarrow \langle n \rangle$-$\mbox{Man}$, where $K$ is the `walking cospan' category $0 \rightarrow 2 \leftarrow 1$. This functor works the same way as the canonical functor $X: \underline{2^n} \rightarrow \mbox{Top}$ defined by Laures, except the category being raised to the power $n$ has an extra object and arrow due to the extra decomposition coming from our cubical $\langle n \rangle$-manifold. We call this $X$ an $n$-\emph{tuple cobordism}. Furthermore, applying the same construction but using the collared version of Laures' functor, we obtain a \emph{collared} $n$-\emph{tuple cobordism} $C: K^n \rightarrow \langle n \rangle$-$\mbox{Man}$; that is, any cubical $\langle n \rangle$-manifold can be `collared'. An example of the result of applying this functor to collar a cubical $\langle 2 \rangle $-manifold is illustrated below.
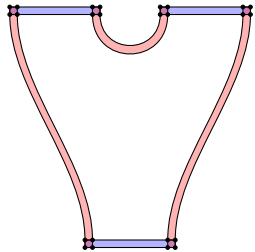
\begin{center}

 \begin{tikzpicture}[every node/.style={coordinate}]
  \node (LL) at (1,0) {};
  \node (LLd) at (1,-0.1) {};
  \node (LLl) at (0.9,0) {};
  \node (LLc) at (0.9,-0.1) {};
  \node (LR) at (2,0) {};
  \node (LRd) at (2,-0.1) {};
  \node (LRr) at (2.1,0) {};
  \node (LRc) at (2.1,-0.1) {};
  \node (UL) at (0,3) {};
  \node (ULl) at (-0.1,3) {};
  \node (ULu) at (0,3.1) {};
  \node (ULc) at (-0.1,3.1) {};
  \node (LU) at (1,3) {};
  \node (LUr) at (1.1,3) {};
  \node (LUu) at (1,3.1) {};
  \node (LUc) at (1.1,3.1) {};
  \node (RU) at (2,3) {};
  \node (RUl) at (1.9,3) {};
  \node (RUu) at (2,3.1) {};
  \node (RUc) at (1.9,3.1) {};
  \node (UR) at (3,3) {};
  \node (URr) at (3.1,3) {};
  \node (URu) at (3,3.1) {};
  \node (URc) at (3.1,3.1) {};

  \filldraw[fill=blue, fill opacity=0.3] (LLc) -- (LRc) -- (LRr) -- (LLl) --cycle;
  \filldraw[fill=blue, fill opacity=0.3] (ULc) -- (LUc) -- (LUr) -- (ULl) --cycle;
  \filldraw[fill=blue, fill opacity=0.3] (RUc) -- (URc) -- (URr) -- (RUl) --cycle;

  \filldraw[fill=red, fill opacity=0.3] (LLc) -- (LLl) .. controls +(90:1) and +(270:1) .. (ULl) -- (ULc) -- (ULu) -- (UL)
  .. controls +(270:1) and +(90:1) .. (LL) -- (LLd) --cycle;

  \filldraw [fill=red, fill opacity=0.3] (LRc) -- (LRr) .. controls +(90:1) and +(270:1) .. (URr) -- (URc) -- (URu) -- (UR)
  .. controls +(270:1) and +(90:1) .. (LR) -- (LRd) --cycle;

  \filldraw [fill=red, fill opacity=0.3] (LUc) -- (LUr) .. controls +(270:0.55) and +(270:0.55) .. (RUl) -- (RUc) -- (RUu) -- (RU)
  .. controls +(270:0.7) and +(270:0.7) .. (LU) -- (LUu) --cycle;

  \draw (UL) -- (ULl) (LL) -- (LLl) (LU) -- (LUr) (RU) -- (RUl) (UR) -- (URr) (LR) -- (LRr);

  \node[dot] (dLL) at (1,0) {};
  \node[dot] (dLLd) at (1,-0.1) {};
  \node[dot] (dLLl) at (0.9,0) {};
  \node[dot] (dLLc) at (0.9,-0.1) {};
  \node[dot] (dLR) at (2,0) {};
  \node[dot] (dLRd) at (2,-0.1) {};
  \node[dot] (dLRr) at (2.1,0) {};
  \node[dot] (dLRc) at (2.1,-0.1) {};
  \node[dot] (dUL) at (0,3) {};
  \node[dot] (dULl) at (-0.1,3) {};
  \node[dot] (dULu) at (0,3.1) {};
  \node[dot] (dULc) at (-0.1,3.1) {};
  \node[dot] (dLU) at (1,3) {};
  \node[dot] (dLUr) at (1.1,3) {};
  \node[dot] (dLUu) at (1,3.1) {};
  \node[dot] (dLUc) at (1.1,3.1) {};
  \node[dot] (dRU) at (2,3) {};
  \node[dot] (dRUl) at (1.9,3) {};
  \node[dot] (dRUu) at (2,3.1) {};
  \node[dot] (dRUc) at (1.9,3.1) {};
  \node[dot] (dUR) at (3,3) {};
  \node[dot] (dURr) at (3.1,3) {};
  \node[dot] (dURu) at (3,3.1) {};
  \node[dot] (dURc) at (3.1,3.1) {};

 \end{tikzpicture}
\end{center}

Notice that if the vertical arrows $s_0X \cap s_1X \rightarrow s_1X \leftarrow t_0X \cap s_1X$ and $s_0X \cap t_1X \rightarrow t_1X \leftarrow t_0X \cap t_1X$ in the above diagram are in fact identity morphisms, everything reduces to a collared cobordism of cobordisms. This is exactly the kind of structure amenable to our main result after using the opposite category, since all relevant pullbacks of these collared cobordisms exist. We thus have
\begin{corollary} There is a bicategory where objects are $(n-2)$-manifolds without boundary, morphisms are $(n-1)$-dimensional collared cobordisms of those manifolds, and 2-morphisms are $n$-dimensional collared cobordisms of cobordisms.
\end{corollary}

\section{Down the road}
Showing $\mbox{Span}_2{\mathcal{C}}$ is a bicategory is just the first step of a larger project. Assuming $\mathcal{C}$ has not only pullbacks but a product, we intend to prove $\mbox{Span}_2{\mathcal{C}}$ is a symmetric monoidal bicategory under the product inherited from $\mathcal{C}$. Verifying the axioms for a symmetric monoidal bicategory (see \cite{Sta}) one-by-one is arduous, so we will adopt a technique of Mike Shulman's for constructing symmetrical monoidal bicategories out of suitable symmetric monoidal double categories \cite{Sch}.

\setlength{\parskip}{0cm} 


\newpage

\begin{thebibliography}{9}

\bibitem{Ben} J. B\'enabou, Introduction to bicategories, \textsl{Reports of the Midwest Category Seminar (Lecture Notes in Mathematics)} {\bf 47}, Springer Berlin Heidelberg, 1967.


\bibitem{Kho} Mikhail Khovanov, Heisenberg algebra and a graphical calculus (2010). Preprint, available at \href{http://arxiv.org/pdf/1009.3295}{http://arxiv.org/pdf/1009.3295}.


\bibitem{MorVic} Jeffrey C. Morton and Jamie Vicary, The categorified Heisenberg algebra I: A combinatorial representation (2012).  Preprint, available at \href{http://arxiv.org/pdf/1207.2054}{http://arxiv.org/pdf/1207.2054}.

\bibitem{Sta} Michael Stay, Compact closed bicategories (2013). Preprint, available at \href{http://arxiv.org/pdf/1301.1053}{http://arxiv.org/pdf/1301.1053}.



\bibitem{Lau} Gerd Laures, On cobordism of manifolds with corners, \textsl{Transactions of the American Mathematical Society} {\bf 352} No.12, (2000), 5668--5672

\bibitem{S-P} Christopher Schommer-Pries, The classification of two-dimensional extended topological field theories (2011), 181.  Preprint, available at \href{http://arxiv.org/pdf/1112.1000}{http://arxiv.org/pdf/1112.1000}.

\bibitem{Mor} Jeffrey C. Morton, Double bicategories and double cospans, \textsl{Journal of Homotopy and Related Structures} {\bf 4} (2009), available at \href{http://arxiv.org/pdf/math/0611930}{http://arxiv.org/pdf/math/0611930}.

\bibitem{Jan} Klaus J\"anich, On the classification of regular $O(n)-$manifolds in terms of their orbit bundles, \textsl{ Proceedings of the Conference on Transformation Groups} (1968), 137--140, Springer-Verlag New York Inc.

\bibitem{Joy} Dominic Joyce, On manifolds with corners (2010), 3--11. Preprint, to appear in  proceedings of \textsl{The Conference on Geometry}. Available at \href{http://arxiv.org/pdf/0910.3518}{http://arxiv.org/pdf/0910.3518}.

\bibitem{Sch} Michael A. Shulman, Constructing symmetric monoidal bicategories (2010). Preprint, available at \href{http://arxiv.org/pdf/1004.0993}{http://arxiv.org/pdf/1004.0993}.


\end{thebibliography}
\end{document}